\documentclass[final]{siamltex}
\usepackage{latexsym,amsfonts,amsmath,graphics}
\usepackage{epsfig}

\newtheorem{remark}{Remark}

\newcommand{\bsk}{\boldsymbol{k}}
\newcommand{\bsl}{\boldsymbol{l}}

\newcommand{\bsx}{\boldsymbol{x}}
\newcommand{\bsh}{\boldsymbol{h}}
\newcommand{\bsg}{\boldsymbol{g}}
\newcommand{\bsz}{\boldsymbol{z}}
\newcommand{\bsy}{\boldsymbol{y}}
\newcommand{\bssigma}{\boldsymbol{\sigma}}
\newcommand{\D}{{\cal D}}

\newcommand{\wal}{{\rm wal}}

\newcommand{\de}{{\rm e}}
\newcommand{\landau}{{\cal O}}
\newcommand{\icomp}{\mathtt{i}}
\newcommand{\bszero}{\boldsymbol{0}}
\newcommand{\rd}{\,\mathrm{d}}
\newcommand{\HH}{\mathcal{H}}
\newcommand{\NN}{\mathbb{N}}
\newcommand{\integer}{\mathbb{Z}}
\newcommand{\RR}{\mathbb{R}}

\newcommand{\LL}{\mathcal{L}}
\renewcommand{\pmod}[1]{\,(\bmod\,#1)}
\newcommand{\EE}{\mathbb{E}}

\makeatletter
\newcommand{\rdots}{\mathinner{\mkern1mu\lower-1\p@\vbox{\kern7\p@\hbox{.}}
\mkern2mu \raise4\p@\hbox{.}\mkern2mu\raise7\p@\hbox{.}\mkern1mu}}
\makeatother

\begin{document}

\title{Explicit constructions of quasi-Monte Carlo rules  for the numerical integration of high dimensional  periodic functions}

\author{Josef Dick\thanks{School of Mathematics and Statistics, UNSW, Sydney, 2052, Australia. ({\tt
josi@maths.unsw.edu.au})}}

\date{}
\maketitle

\begin{abstract}
In this paper we give explicit constructions of point sets in the
$s$ dimensional unit cube yielding quasi-Monte Carlo algorithms
which achieve the optimal rate of convergence of the worst-case
error for numerically integrating high dimensional periodic
functions. In the classical measure $P_{\alpha}$ of the worst-case
error introduced by Korobov the convergence is of
$\landau(N^{-\min(\alpha,d)} (\log N)^{s\alpha-2})$ for every even
integer $\alpha \ge 1$, where $d$ is a parameter of the construction
which can be chosen arbitrarily large and $N$ is the number of
quadrature points. This convergence rate is known to be best
possible up to some $\log N$ factors. We prove the result for the
deterministic and also a randomized setting. The construction is
based on a suitable extension of digital $(t,m,s)$-nets over the
finite field $\integer_b$.
\end{abstract}



\pagestyle{myheadings}
\thispagestyle{plain}
\markboth{Josef DICK}{Explicit constructions of quasi-Monte Carlo rules for periodic integrands}

\section{Introduction}

Korobov~\cite{kor} and independently Hlawka~\cite{hlawka} introduced
a quadrature formula which is suited for numerically integrating
high dimensional periodic functions. More precisely, we want to
approximate the high dimensional integral $\int_{[0,1]^s} f(\bsx)\rd
\bsx$ (where $f$ is assumed to be periodic with period $1$ in each
coordinate) by a quasi-Monte Carlo rule, i.e., an equal weight
quadrature rule $Q_{N,s}(f) = N^{-1} \sum_{n=0}^{N-1} f(\bsx_n)$,
where $\bsx_0,\ldots, \bsx_{N-1} \in [0,1]^s$ are the quadrature
points. Specifically, Korobov and Hlawka suggested using a
quadrature rule of the form $Q_{N,\bsg,s}(f) =
N^{-1}\sum_{n=0}^{N-1} f(\{n\bsg/N\})$, where for a vector of real
numbers $\bsx = (x_1,\ldots, x_s)$ we define $\{\bsx\}$ as the
fractional part of each component of $\bsx$, i.e., $\{x_j\} = x_j -
\lfloor x_j \rfloor = x_j \pmod{1}$ and where $\bsg \in \integer^s$
is an integer vector. The quadrature rule $Q_{N,\bsg,s}$ is called
{\it lattice rule} and $\bsg$ is called the generating vector (of
the lattice rule). The monographs \cite{HuaW, korbook, niesiam,
SJ94} deal partly or entirely with the approximation of such
integrals. (Note that the assumption that the integrand $f$ is
periodic is not really a restriction since there are transformations
which transform non-periodic functions into periodic ones such that
the smoothness of the integrand is preserved, see for example
\cite{SJ94}.)

To analyze the properties of a quadrature rule one considers then
the worst-case error $\sup_{f \in B_\HH} |\int_{[0,1]^s}
f(\bsx)\rd\bsx - Q_{N,s}(f)|$, where $B_\HH$ denotes some class of
functions. In the classical theory the class $\varepsilon_\alpha^s$
of periodic functions has been considered where one demands that the
absolute values of the Fourier coefficients of the function decay
sufficiently fast (see \cite{HuaW, korbook, SJ94, niesiam}). This
leads us to the classical measure of the quality of lattice rules
$P_\alpha = \sup_{f\in\varepsilon^s_{\alpha}} \left|\int_{[0,1]^s}
f(\bsx)\rd\bsx - Q_{N,s}(f)\right|$, which then for a lattice rule
with generating vector $\bsg = (g_1,\ldots, g_s)$ can also be
written as $$P_{\alpha} = P_{\alpha}(\bsg,N) =
\sum_{\bsh\in\integer^s\setminus\{\bszero\}\atop \bsh \cdot\bsg
\equiv 0 \pmod{N}} |\bar{\bsh}|^{-\alpha},$$ where $\bsh =
(h_1,\ldots, h_s)$, $\bsh \cdot \bsg = h_1 g_1 + \cdots + h_s g_s$
and $|\bar{\bsh}| = \prod_{j=1}^s \max(1,|h_j|)$. (Later on in this
paper we prefer to use the more contemporary notation of reproducing
kernel Hilbert spaces, in our case so-called Korobov spaces, (see
Section~\ref{sect_korspace}), but as is well understood (and as is
also shown in Section~\ref{sect_korspace}) the results also apply to
the classical problem.)

By averaging over all generating vectors $\bsg$ several existence
results for good lattice rules which achieve $P_{\alpha} =
\landau(N^{-\alpha}(\log N)^{\alpha s})$ have been shown, see
\cite{HuaW, kor, korbook, nielat, niesiam, SJ94}. By a lower bound
of Sharygin~\cite{shar} this convergence is also known to be
essentially best possible, as he showed that the worst-case error is
at least of order $N^{-\alpha} (\log N)^{s-1}$. But, except for
dimension $s = 2$, no explicit generating vectors $\bsg$ which yield
a small worst-case error are known. For $s \ge 3$ one relies on
computer search to find good generating vectors $\bsg$ and many such
search algorithms have been introduced and analysed, especially
recently, see \cite{kor, SKJ, SKJ2, WSD}.

On the other hand one can of course also use some other quadrature
rule $Q_{N,s}(f) = \sum_{n=0}^{N-1} \omega_n f(\bsx_n)$ to
numerically integrate functions in the class
$\varepsilon^s_{\alpha}$. In this case the worst-case error in the
class $\varepsilon^s_{\alpha}$ for a quadrature rule with weights
$\omega_0,\ldots, \omega_{N-1}$ and points $\{\bsx_0,\ldots,
\bsx_{N-1}\}\subset [0,1)^s$ is given by
\begin{equation}\label{eq_Palphagen}
P_{\alpha}(\{\bsx_0,\ldots, \bsx_{N-1}\}) =  \sum_{n,m=0}^{N-1}
\omega_n \omega_m  \sum_{\bsh \in \integer^s\setminus\{\bszero\}}
\frac{\de^{2\pi \icomp \bsh \cdot
(\bsx_n-\bsx_m)}}{|\bar{\bsh}|^{\alpha}}.
\end{equation}

An explicit construction of such point sets was introduced by
Niederreiter, see \cite[Theorem~5.3]{nie78}, and is called Kronecker
sequence. Here the idea is to choose the quadrature points of the
form $\{\bsz k\}$, $k = 1, 2, \ldots$, where $\bsz$ is an
$s$-dimensional vector of certain irrational numbers (for example
one can choose $\bsz = (\sqrt{p_1},\ldots, \sqrt{p_s})$ where $p_1,
\ldots, p_s$ are distinct prime numbers. Depending on the smoothness
$\alpha$ certain points will be used more than once, see
\cite{nie78}. In practice, problems can occur because of the finite
precision of computers making it impossible to use points whose
coordinates are all irrational numbers.

Another construction of quadrature rules is due to
Smolyak~\cite{smo} and is nowadays called sparse grid, see also
\cite{gerst}. Those quadrature rules are sums over certain products
of differences of one-dimensional quadrature rules. In principle any
one-dimensional quadrature rule can be chosen as a basis, leading to
different quadrature rules. In many cases the weights $\omega_n$ of
such quadrature rules  are not known explicitly but can be
precomputed. But even if the underlying one-dimensional quadrature
rule has only positive weights, it is possible that some weights in
Smolyak's quadrature rules are negative, which can have a negative
impact on the stability of the quadrature formula. In general,
quadrature formulae for which all weights are equal and
$\sum_{n=0}^{N-1} \omega_n = 1$, that is, $\omega_n = N^{-1}$ for
all $n = 0,\ldots, N-1$, are to be preferred. As mentioned above,
such quadrature rules are called quasi-Monte Carlo rules, to which
we now switch for the remainder of the paper.

As the weights for quasi-Monte Carlo rules are given by $N^{-1}$ the
focus lies on the choice  of the quadrature points. Constructions of
quadrature points have been introduced with the aim to distribute
the points as evenly as possible over the unit cube. An explicit
construction of well distributed point sets in the unit cube has
been introduced by Sobol~\cite{sob67}. A similar construction was
established by Faure~\cite{faure} before Niederreiter~\cite{nie87}
(see also \cite{niesiam}) introduced the general concept of
$(t,m,s)$-nets and $(t,s)$-sequences and the construction scheme of
digital $(t,m,s)$-nets and digital $(t,s)$-sequences. For such point
sets it has been shown that the star discrepancy (which is a measure
of the distribution properties of a point set) is
$\landau(N^{-1}(\log N)^{s-1})$, see \cite{niesiam}. From this
result it follows that those point sets yield quasi-Monte Carlo
algorithms which achieve a convergence of $\landau(N^{-2}(\log
N)^{2s-2})$ for functions in the class $\varepsilon^s_{\alpha}$ for
all $\alpha \ge 2$. This result holds in the deterministic and
randomized setting.

For smoother functions though, i.e., larger values of $\alpha$ in
the class $\varepsilon_{\alpha}^s$, one can expect higher order
convergence. For example, if the partial derivatives up to order two
are square integrable then one would expect an integration error of
$\landau(N^{-4} (\log N)^{c(s)})$, for some $c(s) > 0$ depending
only on $s$, in the function class $\varepsilon_{\alpha}^s$, and in
general, if the mixed partial derivatives up to order $\alpha/2$
exist and are square integrable then one would expect an integration
error in $\varepsilon_{\alpha}^s$ of $\landau(N^{-\alpha}(\log
N)^{c(s,\alpha)})$, for some $c(s,\alpha) > 0$ depending only on $s$
and $\alpha$. But until now $(t,m,s)$-nets and $(t,s)$-sequences
have only been shown to yield a convergence of at best
$\landau(N^{-2}(\log N)^{2s-2})$ (or $\landau(N^{-3+\delta})$ for
any $\delta >0$ if one uses a randomization method called
scrambling, see \cite{owen}) in $\varepsilon_{\alpha}^s$, even if
the integrands satisfy stronger smoothness assumptions.

In this paper we show that a modification of digital $(t,m,s)$-nets
and digital $(t,s)$-sequences introduced by
Niederreiter~\cite{nie87, niesiam} yields point sets which achieve
the optimal rate of convergence of the worst-case error $P_{2\alpha}
= \landau(N^{-2\min(\alpha,d)}(\log N)^{2s\alpha -2})$ for any
integer $\alpha \ge 1$ and where $d \in \NN$ is a parameter of the
construction which can be chosen arbitrarily large. We too use the
digital construction scheme introduced by Niederreiter~\cite{nie87,
niesiam} for the construction of $(t,m,s)$-nets and
$(t,s)$-sequences, but our analysis of the worst-case error shows
that the $t$-value does not provide enough information about the
point set. Hence we generalize the definition of digital
$(t,m,s)$-nets and digital $(t,s)$-sequences to suit our needs. This
leads us to the definition of digital $(t,\alpha,\beta,m,s)$-nets
and digital $(t,\alpha,\beta,s)$-sequences. For $\alpha = \beta = 1$
those definitions reduce to the case introduced by Niederreiter, but
are different for $\alpha > 1$. Subsequently we prove that
quasi-Monte Carlo rules based on digital $(t,\alpha,\beta,m,s)$-nets
and digital $(t,\alpha,\beta,s)$-sequences achieve the optimal rate
of convergence. Further we give explicit constructions of digital
$(t,\alpha,\min(\alpha,d),m,s)$-nets and digital $(t,\alpha,
\min(\alpha,d),s)$-sequences, where $d \in \NN$ is a parameter of
the construction which can be chosen arbitrarily large.

Digital $(t,2,2,m,s)$-nets and digital $(t,2,2,s)$-sequences over
$\integer_b$ (i.e. where $\alpha = \beta = 2$) can also be used for
non-periodic function spaces where one uses randomly shifted and
then folded point sets using the baker's transformation (see
\cite{CDLP}). Our analysis and error bounds for $\alpha = 2$  here
also apply for the case considered in \cite{CDLP} (with different
constants though), hence yielding useful constructions also for
non-periodic function spaces where one uses the baker's
transformation. Using a digital $(t,\alpha,m,s)$-net with a
scrambling algorithm (see \cite{owen}) on the other hand does not
improve the performance in non-periodic spaces compared to
$(t,m,s)$-nets.

In the following we summarize some properties of the quadrature rules:
\begin{itemize}
\item The quadrature rules introduced in this paper are equal weight quadrature rules which achieve the optimal rate of convergence up to some $\log N$ factors and we show the result for deterministic and randomly digitally shifted quadrature rules. The upper bound for the randomized quadrature rules even improves upon the best known upper bound (more precisely, the power of the $\log N$ factor) for lattice rules for the worst-case error in $\varepsilon_\alpha^s$ for all dimensions $s \ge 2$ and even integers $\alpha \ge 2$ (compare Corollary~\ref{cor_bounderand} to Theorem~2 in \cite{nielat}).
\item The construction of the underlying point set is explicit.
\item They automatically adjust themselves to the optimal rate of convergence in the class $\varepsilon_{2\alpha}^s$ as long as $\alpha$ is an integer such that $\alpha \le d$, where $d$ is a parameter of the construction which can be chosen arbitrarily large.
\item The underlying point set is extensible in the dimension as well as in the number of points, i.e., one can always add some coordinates or points to an existing point set such that the quality of the point set is preserved.
\item Tractability and strong tractability results (see \cite{SW98}) can be obtained for weighted Korobov spaces.
\end{itemize}

The outline of the paper is as follows. In the next section we
introduce the necessary tools, namely Walsh functions, the digital
construction scheme upon which the construction of the point set is
based on and Korobov spaces. Further we also introduce the
worst-case error in those Korobov spaces and we give a
representation of this worst-case error for digital nets in terms of
the Walsh coefficients of the reproducing kernel. In
Section~\ref{sec_talpha} we give the definition of digital
$(t,\alpha,\beta,m,s)$-nets and digital
$(t,\alpha,\beta,s)$-sequences. Further we prove some propagation
rules for those digital nets and sequences. In
Section~\ref{sec_talphacons} we give explicit constructions of
digital $(t,\alpha,\beta,m,s)$-nets and digital
$(t,\alpha,\beta,s)$-sequences and we prove some upper bounds on the
$t$-value. We then show, Section~\ref{sec_bound}, that quasi-Monte
Carlo rules based on those digital nets and sequences achieve the
optimal rate of convergence of the worst-case error in the Korobov
spaces. The results are based on entirely deterministic point sets.
Section~\ref{sec_boundrand} finally deals with randomly digitally
shifted digital $(t,\alpha,\beta,m,s)$-nets and
$(t,\alpha,\beta,s)$-sequences and we show similar results for the
mean square worst-case error in the Korobov space for this setting.
The Appendix is devoted to the analysis of the Walsh coefficients of
the Walsh series representation of $B_{2\alpha}(|x-y|)$, where
$B_{2\alpha}$ is the Bernoulli polynomial of  degree $2\alpha$. In
the last section we give a concrete example of a digital
$(t,\alpha,\alpha,m,s)$-net where we compute the $t$-value by hand.

\section{Preliminaries}\label{sect_pre}

In this section we introduce the necessary tools for the analysis of
the worst-case error  and the construction of the point sets. In the
following let $\NN$ denote the set of natural numbers and let
$\NN_0$ denote the set of non-negative integers.

\subsection{Walsh functions}

In the following we define Walsh functions in base $b\ge 2$, which are the main tool of analyzing the worst-case error. First we give the definition for the one-dimensional case.

\begin{definition}\rm
Let $b \ge 2$ be an integer and represent $k \in \NN_0$ in base $b$, $k = \kappa_{a-1} b^{a-1} + \cdots + \kappa_0$ with $\kappa_i \in \{0,\ldots, b-1\}$. Further let $\omega_b = \de^{2\pi\icomp/b}$. Then the $k$-th Walsh function $_b\wal_k: [0,1)\rightarrow \{1,\omega_b,\ldots, \omega_b^{b-1}\}$ in base $b$ is given by $$_b\wal_k(x) = \omega_b^{x_1 \kappa_0 + \cdots + x_a \kappa_{a-1}},$$ for $x \in [0,1)$ with base $b$ representation $x = x_1 b^{-1} + x_2 b^{-2} + \cdots$ (unique in the sense that infinitely many of the $x_i$ are different from $b-1$).
\end{definition}

\begin{definition}\rm
For dimension $s \ge 2$, $\bsx = (x_1,\ldots, x_s) \in [0,1)^s$ and $\bsk = (k_1,\ldots, k_s) \in \NN_0^s$ we define $_b\wal_{\bsk}:[0,1)^s \rightarrow \{1,\omega_b,\ldots, \omega_b^{b-1}\}$ by $$_b\wal_{\bsk}(\bsx) = \prod_{j=1}^s\; _b\wal_{k_j}(x_j).$$
\end{definition}
As we will always use Walsh functions in base $b$ we will in the following often write $\wal$ instead of $_b\wal$.

We introduce some notation. By $\oplus$ we denote the digit-wise addition modulo $b$, i.e., for $x = \sum_{i=w}^\infty x_i b^{-i}$ and $y = \sum_{i=w}^\infty y_i b^{-i}$ we define $$x \oplus y = \sum_{i = w}^\infty z_i b^{-i},$$ where $z_i \in \{0,\ldots, b-1\}$ is given by $z_i \equiv x_i + y_i \pmod{b}$ and let $\ominus$ denote the digit-wise subtraction modulo $b$. In the same manner we also  define a digit-wise addition and digit-wise subtraction for non-negative integers based on the $b$-adic expansion. For vectors in $[0,1)^s$ or $\NN_0^s$ the operations $\oplus$ and $\ominus$ are carried out component-wise. Throughout the paper we always use base $b$ for the operations $\oplus$ and $\ominus$. Further we call $x\in [0,1)$ a $b$-adic rational if it can be written in a finite base $b$ expansion.

In the following proposition we summarize some basic properties of Walsh functions.

\begin{proposition}\label{prop1}
\begin{enumerate}
\item For all $k,l \in \NN_0$ and all $x,y \in [0,1)$, with the restriction that if $x, y$ are not $b$-adic rationals then $x \oplus y$ is not allowed to be a $b$-adic rational,  we have $$\wal_k(x) \cdot \wal_l(x) = \wal_{k\oplus l}(x), \quad \wal_k(x)\cdot\wal_k(y) = \wal_k(x\oplus y).$$
\item We have $$\int_0^1 \wal_0(x) \rd x = 1 \quad \mbox{and} \quad \int_0^1 \wal_k(x) = 0 \mbox{ if } k >0.$$
\item For all $\bsk,\bsl \in \NN_0^s$ we have the following orthogonality properties: $$\int_{[0,1)^s} \wal_{\bsk}(\bsx) \wal_{\bsl}(\bsx)\rd\bsx = \left\{\begin{array}{ll} 1, & \mbox{if } \bsk = \bsl, \\ 0, & \mbox{otherwise}.   \end{array} \right.$$
\item For any $f \in \LL_2([0,1)^s)$ and any $\bssigma \in [0,1)^s$ we have $$\int_{[0,1)^s} f(\bsx \oplus \bssigma)\rd\bsx = \int_{[0,1)^s} f(\bsx)\rd\bsx.$$
\item For any integer $s \ge 1$ the system $\{\wal_{\bsk}: \bsk = (k_1,\ldots, k_s), k_1,\ldots, k_s \ge 0\}$ is a complete orthonormal system in $\LL_2([0,1)^s)$.
\end{enumerate}
\end{proposition}

The proofs of 1.-3. are straightforward and for a proof of the remaining items see \cite{chrest} or \cite{walsh} for more information.

\subsection{The digital construction scheme}

The construction of the point set used here is based on the digital
construction scheme introduced by Niederreiter, see \cite{niesiam}.

\begin{definition}\rm\label{def_digcons}
Let integers $m,s\ge 1$ and $b \ge 2$ be given. Let $R_b$ be a
commutative ring with identity such that  $|R_b| = b$ and let
$\integer_b = \{0,\ldots, b-1\}$. Let $C_1,\ldots, C_s \in R_b^{m
\times m}$ with $C_j = (c_{j,k,l})_{1\le k,l\le m}$. Further, let
$\psi_l:\integer_b \rightarrow R_b$ for $l = 0,\ldots, m-1$ and
$\mu_{j,k}: R_b \rightarrow \integer_b$ for $j = 1,\ldots, s$ and $k
= 1,\ldots, m$ be bijections.

For $n = 0,\ldots, b^m-1$ let $n = \sum_{l=0}^{m-1}a_l(n) b^l$, with
all $a_l(n) \in \integer_b$, be the base $b$ digit expansion of $n$.
Let $\vec{n} = (\psi_0(a_0(n)),\ldots, \psi_{m-1}(a_{m-1}(n)))^T$
and let $\vec{y}_j = (y_{j,1},\ldots, y_{j,m})^T = C_j \vec{n}$ for
$j = 1,\ldots, s$. Then we define $x_{j,n} = \mu_{j,1}(y_{j,1})
b^{-1} + \cdots + \mu_{j,m}(y_{j,m}) b^{-m}$ for $j = 1,\ldots, s$
and $n = 0,\ldots, b^m-1$ and the $n$-th point $\bsx_n$ is then
given by $\bsx_n = (x_{1,n},\ldots, x_{s,n})$. The point set
$\{\bsx_0,\ldots, \bsx_{b^m-1}\}$ is called a digital net (over
$R_b$) (with generating matrices $C_1,\ldots, C_s$).

For $m = \infty$ we obtain a sequence $\{\bsx_0,\bsx_1,\ldots\}$,
which is called a digital sequence (over $R_b$) (with generating
matrices $C_1,\ldots, C_s$).
\end{definition}

Niederreiter's concept of a digital $(t,m,s)$-net and a digital $(t,s)$-sequence will appear as a special case in Section~\ref{sec_talpha}. Apart from Section~\ref{sec_talpha} and Section~\ref{sec_talphacons}, where we state the results using Definition~\ref{def_digcons} in the general form, we use only a special case of Definiton~\ref{def_digcons}, where we assume that $b$ is a prime number, we choose $R_b$ the finite field $\integer_b$ and the bijections $\psi_l$ and $\mu_{j,k}$ from $\integer_b$ to $\integer_b$ are all chosen to be the identity map.

We remark that throughout the paper when Walsh functions $\wal$, digit-wise addition $\oplus$, digit-wise subtraction $\ominus$ or digital nets are used in conjunction with each other we always use the same base $b$ for each of those operations.

\subsection{Korobov space}\label{sect_korspace}

Historically the function class $\varepsilon^s_{\alpha}$ has been used. In this paper we use a more contemporary notation by replacing the function class $\varepsilon^s_{\alpha}$ with a reproducing kernel Hilbert space $\HH_\alpha$ called Korobov space. The worst-case error expression (\ref{eq_Palphagen}) will almost be the same for both function classes and hence the results apply for both cases.

A reproducing kernel Hilbert space $\HH$ over $[0,1)^s$ is a Hilbert space with inner product $\langle \cdot, \cdot \rangle$ which allows a function $K: [0,1)^s \rightarrow \RR$ such that $K(\cdot,\bsy) \in \HH$, $K(\bsx,\bsy) = K(\bsy,\bsx)$ and $\langle f,K(\cdot,\bsy)\rangle = f(\bsy)$ for all $\bsx, \bsy \in [0,1)^s$ and all $f \in \HH$. For more information on reproducing kernel Hilbert spaces see \cite{aron}, for more information on reproducing kernel Hilbert spaces in the context of numerical integration see for example \cite{DP,SW98}.

The Korobov space $\HH_\alpha$ is a reproducing kernel Hilbert space of periodic functions. Its reproducing kernel is given by
\begin{equation*}
K_{\alpha}(\bsx,\bsy) = \sum_{\bsh \in \integer^s} \frac{\de^{2 \pi \icomp \bsh \cdot (\bsx - \bsy)}}{|\bar{\bsh}|^{2\alpha}},
\end{equation*}
where $\alpha > 1/2$ and $|\bar{\bsh}| = \prod_{j=1}^s \max(1,|h_j|)$. The inner product in the space $\HH_\alpha$ is given by
\begin{equation}\label{eq_inprod}
\langle f,g\rangle_\alpha = \sum_{\bsh \in \integer^s} |\bar{\bsh}|^{2\alpha} \hat{f}(\bsh) \hat{g}(\bsh),
\end{equation}
 where $$\hat{f}(\bsh) = \int_{[0,1)^s} f(\bsx)\de^{-2\pi\icomp \bsh\cdot \bsx}\rd\bsx$$ are the Fourier coefficients of $f$. The norm is given by $\|f\|_\alpha = \langle f, f\rangle_\alpha^{1/2}$.

Note that for $\alpha$ a  natural number and any $x \in (0,1)$ we have
$$B_{2\alpha}(x) = \frac{(-1)^{\alpha + 1}(2\alpha)!}{(2\pi)^{2\alpha}}\sum_{h \neq 0} \frac{\de^{2 \pi \icomp h x}}{|h|^{2\alpha}},$$
where $B_{2\alpha}$ is the Bernoulli polynomial of degree $2\alpha$. Hence, for $\alpha$ a natural number we can write
\begin{equation*}
K_{\alpha}(\bsx,\bsy) = \prod_{j=1}^s \left(1 + \sum_{h \neq 0} \frac{\de^{2\pi \icomp h (x_j - y_j)}}{|h|^{2\alpha}} \right)
= \prod_{j=1}^s \left(1 - (-1)^{\alpha} \frac{(2\pi)^{2\alpha}}{(2\alpha)!} B_{2\alpha}(|x_j - y_j|) \right).
\end{equation*}
Let now
\begin{equation}\label{Kdef}
K_{\alpha}(x,y) = 1 + \sum_{h \neq 0} \frac{\de^{2\pi \icomp h (x-y)}}{|h|^{2\alpha}} = 1-(-1)^\alpha \frac{(2\pi)^{2\alpha}}{(2\alpha)!}B_{2\alpha}(|x-y|).
\end{equation}
Then we have $$K_{\alpha}(\bsx,\bsy) = \prod_{j=1}^s K_{\alpha}(x_j,y_j),$$ where $\bsx = (x_1,\ldots, x_s)$ and $\bsy = (y_1,\ldots, y_s)$. Hence the Korobov space is a tensor product of one-dimensional reproducing kernel Hilbert spaces.

Though $\alpha > 1/2$ can in general be any real number we restrict ourselves to integers $\alpha \ge 1$ for most of this paper. The bounds on the integration error for $\HH_\alpha$ with $\alpha \ge 1$ a real number still apply when one replaces $\alpha$ with $\lfloor \alpha \rfloor$, as in this case the unit ball of $\HH_\alpha$ given by $\{f \in \HH_\alpha: \|f\|_\alpha \le 1\}$ is contained in the unit ball $\{f \in \HH_{\lfloor \alpha \rfloor}: \|f\|_{\lfloor \alpha \rfloor} \le 1\}$ of $\HH_{\lfloor \alpha \rfloor}$ as $\|f\|_{\lfloor \alpha \rfloor} \le \|f\|_{\alpha}$. Hence it follows that integration in the space $\HH_\alpha$ is easier than integration in the space $\HH_{\lfloor \alpha\rfloor}$.

In general, the worst-case error $e(P,\HH)$ for multivariate integration in a normed space $\HH$ over $[0,1]^s$ with norm $\|\cdot\|$ using a point set $P$ is given by $$e(P,\HH) = \sup_{f \in \HH, \|f\| \le 1} \left|\int_{[0,1]^s} f(\bsx)\rd\bsx - Q_{P}(f)\right|,$$ where $Q_P(f) = N^{-1} \sum_{\bsx \in P}f(\bsx)$ and $N = |P|$ is the number of points in $P$. If $\HH$ is a reproducing kernel Hilbert space with reproducing kernel $K$ we will write $e(P,K)$ instead of $e(P,\HH)$. It is known that (see for example \cite{SW98})
\begin{equation}\label{wcekernel}
e^2(P,K) = \int_{[0,1)^{2s}} K(\bsx,\bsy)\rd \bsx \rd \bsy - \frac{2}{N} \sum_{n=0}^{N-1} \int_{[0,1)^s} K(\bsx_n,\bsy) \rd \bsy + \frac{1}{N^2} \sum_{n,l=0}^{N-1} K(\bsx_n,\bsx_l),
\end{equation}
where $P = \{\bsx_0,\ldots, \bsx_{N-1}\}$. Hence for the Korobov space $\HH_\alpha$ we obtain
\begin{equation}\label{eq_wceker}
e^2(P,K_\alpha) = - 1 + \frac{1}{N^2} \sum_{n,h=0}^{N-1} K_\alpha(\bsx_n,\bsx_h).
\end{equation}
Therefore it follows that $e^2(P,K_\alpha) = P_{2\alpha}$ and hence
our results also apply to the classical setting introduced by
Korobov~\cite{kor}.

It follows from Proposition~\ref{prop1} that $K_\alpha$ can be represented by a Walsh series, i.e., let
\begin{equation}\label{eq_walshker}
K_\alpha(\bsx,\bsy) = \sum_{\bsk,\bsl \in \NN_0^s} r_{b,\alpha}(\bsk,\bsl) \wal_{\bsk}(\bsx)\overline{\wal_{\bsl}(\bsy)},
\end{equation}
 where $$r_{b,\alpha}(\bsk,\bsl) = \int_{[0,1)^{2s}} K_\alpha(\bsx,\bsy) \overline{\wal_{\bsk}(\bsx)} \wal_{\bsl}(\bsy) \rd \bsx \rd\bsy.$$ As the kernel $K_\alpha$ is a product of one-dimensional kernels it follows that $r_{b,\alpha}(\bsk,\bsl) = \prod_{j=1}^s r_{b,\alpha}(k_j,l_j)$, where $\bsk = (k_1,\ldots, k_s)$ and $\bsl = (l_1,\ldots, l_s)$ and $$r_{b,\alpha}(k,l) = \int_0^1\int_0^1 K_\alpha(x,y) \overline{\wal_k(x)} \wal_l(y)\rd x \rd y.$$

For a digital net with generating matrices $C_1,\ldots, C_s$ let $\D = \D(C_1,\ldots, C_s)$ be the dual net given by $$\D = \{\bsk \in \NN_0^s\setminus \{\bszero\}: C_1^T \vec{k}_1 + \cdots + C_s^T \vec{k}_s = \vec{0}\},$$ where for $\bsk = (k_1,\ldots, k_s)$ with $k_j = \kappa_{j,0} + \kappa_{j,1} b + \cdots$ we set $\vec{k}_j = (\kappa_{j,0},\ldots, \kappa_{j,m-1})^T$. Further, for $\emptyset \neq u \subseteq \{1,\ldots, s\}$ let $\D_u = \D((C_j)_{j\in u})$. We have the following theorem.

\begin{theorem}\label{th_wcer}
Let $C_1,\ldots, C_s \in \integer_b^{m \times m}$ be the generating matrices of a digital net $P_{b^m}$ and let $\D$ denote the dual net. Then for any $\alpha > 1/2$ the square worst-case error in $\HH_\alpha$ is given by $$e^2(P_{b^m},K_\alpha) = \sum_{\bsk,\bsl \in \D} r_{b,\alpha}(\bsk,\bsl).$$
\end{theorem}

\begin{proof}
From (\ref{eq_wceker}) and (\ref{eq_walshker}) it follows that $$e^2(P_{b^m},K_\alpha) = - 1 + \sum_{\bsk,\bsl \in \NN_0^s} r_{b,\alpha}(\bsk,\bsl) \frac{1}{b^{2m}} \sum_{\bsx,\bsy \in P_{b^m}}  \wal_{\bsk}(\bsx) \overline{\wal_{\bsl}(\bsy)}.$$ In \cite{DP} it was shown that $$\frac{1}{b^{m}} \sum_{\bsx \in P_{b^m}} \wal_{\bsk}(\bsx) = \left\{\begin{array}{ll} 1 & \mbox{if } \bsk \in \D\cup \{\bszero\}, \\ 0 & \mbox{otherwise}. \end{array}\right.$$ Hence we have $$e^2(P_{b^m},K_\alpha) = -1 +  \sum_{\bsk,\bsl \in \D\cup \{\bszero\}} r_{b,\alpha}(\bsk,\bsl).$$

In the following we will show that $r_{b,\alpha}(\bszero,\bszero) = 1$ and $r_{b,\alpha}(\bszero,\bsk) = r_{b,\alpha}(\bsk,\bszero) = 0$ if $\bsk \neq \bszero$ from which the result then follows. Note that it is enough to show those identities for the one dimensional case. We have $\wal_0(x) = 1$ for all $x \in [0,1)$ and hence
\begin{eqnarray*}
r_{b,\alpha}(0,k) & =& \int_0^1\int_0^1 (1 + \sum_{h\in
\integer\setminus\{0\}} |h|^{-2\alpha} \de^{2\pi\icomp h (x-y)})
\wal_k(y) \rd x \rd y \\ &=& \int_0^1 \wal_k(y) \rd y +  \int_0^1
\sum_{h\in \integer\setminus\{0\}} |h|^{-2\alpha}\int_0^1
\de^{2\pi\icomp h x}\rd x \; \de^{-2\pi\icomp h y} \wal_k(y) \rd y \\
&=&  \int_0^1 \wal_k(y) \rd y.
\end{eqnarray*}
It now follows from Proposition~\ref{prop1} that $r_{b,\alpha}(0,0)
= 1$ and $r_{b,\alpha}(0,k) = 0$ for $k > 0$. The result for
$r_{b,\alpha}(k,0)$ can be obtained in the same manner. Hence the
result follows.
\end{proof}

In the following lemma we obtain a formula for the Walsh coefficients $r_{b,\alpha}$.
\begin{lemma}\label{lem_rgen}
Let $b \ge 2$ be an integer and let $\alpha > 1/2$ be a real number. The Walsh coefficients $r_{b,\alpha}(k,l)$ for $k,l \in\NN$  are given by  $$r_{b,\alpha}(k,l) = \sum_{h \in \integer\setminus\{0\}} \frac{\overline{\beta_{h,k}} \beta_{h,l}}{|h|^{2\alpha}},$$ where $\beta_{h,k} = \int_0^1 \de^{-2\pi\icomp h x} \wal_k(x) \rd x$.
\end{lemma}

\begin{proof}
We have
\begin{eqnarray*}
r_{b,\alpha}(k,l) &=& \int_0^1\int_0^1 \sum_{h \in \integer\setminus\{0\}} |h|^{-2\alpha} \de^{2\pi\icomp h (x-y)} \overline{\wal_k(x)}\wal_l(y)\rd x \rd y \\ & = &  \sum_{h \in \integer\setminus\{0\}} |h|^{-2\alpha} \int_0^1  \de^{2\pi\icomp h x} \overline{\wal_k(x)} \rd x  \int_0^1  \de^{-2\pi\icomp h y} \wal_l(y) \rd y.
\end{eqnarray*}
The result follows.
\end{proof}

It is difficult to calculate the exact value of $r_{b,\alpha}(k,l)$
in general, but for our purposes it is enough to obtain an upper
bound. Note that $r_{b,\alpha}(k,k)$ is a non-negative real number.
\begin{lemma}\label{lem_rr}
Let $b \ge 2$ be an integer and let $\alpha > 1/2$ be a real number. The Walsh coefficients $r_{b,\alpha}(k,l)$ for $k,l \in\NN$  are bounded by  $$|r_{b,\alpha}(k,l)|^2 \le r_{b,\alpha}(k,k) r_{b,\alpha}(l,l).$$
\end{lemma}

\begin{proof}
Using Lemma~\ref{lem_rgen} we obtain
\begin{eqnarray*}
|r_{b,\alpha}(k,l)|^2 \le \left(\sum_{h\in\integer\setminus\{0\}} \frac{|\beta_{h,k}| |\beta_{h,l}|}{|h|^{2\alpha}}  \right)^2 & \le & \sum_{h\in\integer\setminus\{0\}} \frac{|\beta_{h,k}|^2}{|h|^{2\alpha}} \sum_{h\in\integer\setminus\{0\}} \frac{|\beta_{h,l}|^2}{|h|^{2\alpha}} \\ & = & r_{b,\alpha}(k,k) r_{b,\alpha}(l,l).
\end{eqnarray*}
The result follows.
\end{proof}

In the following we will write $r_{b,\alpha}(k)$ instead of $r_{b,\alpha}(k,k)$ and also $r_{b,\alpha}(\bsk)$ instead of $r_{b,\alpha}(\bsk,\bsk)$.

\begin{lemma}\label{lem_wcesqrt}
Let $C_1,\ldots, C_s \in \integer_b^{m \times m}$ be the generating matrices of a digital net $P_{b^m}$ and let $\D$ denote the dual net. Then for any natural number $\alpha$ the worst-case error in $\HH_\alpha$ is bounded by $$e(P_{b^m},K_\alpha) \le  \sum_{\bsk \in \D} \sqrt{r_{b,\alpha}(\bsk)}.$$
\end{lemma}
\begin{proof}
From Theorem~\ref{th_wcer} and Lemma~\ref{lem_rr} it follows that $$e^2(P_{b^m},K_\alpha) \le \sum_{\bsk,\bsl \in \D} |r_{b,\alpha}(\bsk,\bsl)| \le \left(\sum_{\bsk \in \D} \sqrt{r_{b,\alpha}(\bsk,\bsk)}\right)^2$$ and hence the result follows.
\end{proof}

 For $\alpha \ge 1$ a natural number we can write the reproducing kernel in terms of Bernoulli polynomials of degree $2\alpha$. Then  for $k \ge 1$ we have
$$r_{b,\alpha}(k) = (-1)^{\alpha+1} \frac{(2\pi)^{2\alpha}}{(2\alpha)!} \int_0^1 \int_0^1 B_{2\alpha}(|x-y|) \overline{\wal_{k}(x)} \wal_k(y) \rd x \rd y.$$
Note that the Bernoulli polynomials of even degree $2\alpha$ are of the form
$$B_{2\alpha}(x) = c_{\alpha} x^{2\alpha} + c_{\alpha-1} x^{2(\alpha -1)} + \cdots + c_0 + c x^{2\alpha -1},$$ for some rational numbers
$c_{\alpha}, \ldots, c_0, c$ with $c_\alpha, c \neq 0$. Let
\begin{equation}\label{eq_Ijk}
I_j(k) = \int_0^1 \int_0^1 |x-y|^j \overline{\wal_{k}(x)} \wal_{k}(y) \rd x \rd y.
\end{equation}
As mentioned above, $r_{b,\alpha}(k)$ is a real number such that $r_{b,\alpha}(k) \ge 0$ for all $k \ge 1$ and $\alpha > 1/2$, hence it follows that  for any natural number $\alpha$ we have $$r_{b,\alpha}(k) \le \frac{(2\pi)^{2\alpha}}{(2\alpha)!} \left(|c_{\alpha} I_{2\alpha}(k) | + |c_{\alpha-1}I_{2(\alpha-1)}(k)| + \cdots + |c_0 I_0(k)| +
|c I_{2\alpha-1}| \right).$$

Using Lemma~\ref{lem_Ijeven} and Lemma~\ref{lem_Iodd} from the Appendix we obtain the following lemma.
\begin{lemma}\label{lem_boundr}
Let $b,\alpha \in \NN$ with $b \ge 2$. For $k \in \NN$ with $k = \kappa_1 b^{a_1-1} +\cdots + \kappa_\nu b^{a_\nu-1}$ where $\nu \ge 1$, $\kappa_1,\ldots, \kappa_\nu \in \{1,\ldots, b-1\}$ and $1 \le a_\nu < \cdots < a_1$ let $q_{b,\alpha}(k) = b^{-a_1-\cdots - a_{\min(\nu,\alpha)}}$. Then for any natural number $\alpha$ and any natural number $b \ge 2$ there exists a constant $C_{b,\alpha} >0$ which depends only on $b$ and $\alpha$ such that $$r_{b,\alpha}(k) \le C^2_{b,\alpha}\; q^2_{b,\alpha}(k) \quad \mbox{for all } k\ge 1.$$
\end{lemma}

Let now $q_{b,\alpha}(0) = 1$. For $\bsk = (k_1,\ldots, k_s) \in \NN_0^s$  we define $q_{b,\alpha}(\bsk) = \prod_{j=1}^s q_{b,\alpha}(k_j)$.  We have the following lemma.
\begin{lemma}\label{lem_q}
Let $m \ge 1$, $b\ge 2$ and $\alpha \ge 2$ be natural numbers and let $\D^\ast_{b^m, u} = \D_u \cap \{1,\ldots,b^m-1\}^{|u|}$.  Then we have
\begin{eqnarray*}
\lefteqn{ \sum_{\bsk \in \D} \sqrt{r_{b,\alpha}(\bsk)} } \\  & \le & \sum_{\emptyset \neq u \subseteq \{1,\ldots, s\}} (1 + b^{-\alpha m}C_{b,\alpha}(\alpha + b^{-2}))^{s-|u|}  C_{b,\alpha}^{|u|} (1+ \alpha +b^{-2})^{|u|} Q^\ast_{b,m,u,\alpha}(C_1,\ldots,C_s) \\  &&  + (1 + b^{-\alpha m} C_{b,\alpha}(\alpha + b^{-2}) )^s -1,
\end{eqnarray*}
where $C_{b,\alpha}$ is the constant from Lemma~\ref{lem_boundr} and where $$Q^\ast_{b,m,u,\alpha}(C_1,\ldots, C_s) = \sum_{\bsk \in \D^\ast_{b^m,u}} q_{b,\alpha}(\bsk).$$
\end{lemma}

\begin{proof}
Every $\bsk \in \NN_0^s$ can be uniquely written in the form $\bsk =
\bsh + b^m \bsl$ with $\bsh \in \{0,\ldots,b^m-1\}^s$ and $\bsl \in
\NN_0^s$. Let $\D_{b^m} = \D \cap \{0,\ldots,b^m-1\}^s$. Then we
have
$$\sum_{\bsk \in \D} \sqrt{r_{b,\alpha}(\bsk)} = \sum_{\bsl \in
\NN_0^s\setminus\{\bszero\}} \sqrt{r_{b,\alpha}(b^m\bsl)} +
\sum_{\bsh \in \D_{b^m}}\sum_{\bsl \in \NN_0^s}\sqrt{
r_{b,\alpha}(\bsh + b^m \bsl)}.$$ For the first sum we have $$
\sum_{\bsl \in \NN_0^s\setminus\{\bszero\}}
\sqrt{r_{b,\alpha}(b^m\bsl)} = -1+  \sum_{\bsl \in \NN_0^s}
\sqrt{r_{b,\alpha}(b^m\bsl)} = -1 + \left(\sum_{l=0}^\infty \sqrt{
r_{b,\alpha}(b^m l)} \right)^s.$$ By using Lemma~\ref{lem_rbm} from
the Appendix and Lemma~\ref{lem_boundr} we obtain that $$ \sum_{l =
0}^\infty \sqrt{r_{b,\alpha}(b^m l)} = 1 + b^{-\alpha m}
\sum_{l=1}^\infty \sqrt{r_{b,\alpha}(l)} \le  1 + b^{-\alpha m}
C_{b,\alpha} \sum_{l=1}^\infty q_{b,\alpha}(l).$$ We need to show
that $ \sum_{l=1}^\infty q_{b,\alpha}(l) \le \alpha + b^{-2}$. Let
$l = l_1 b^{c_1-1} + \cdots + l_\nu b^{c_\nu-1}$ for some $\nu \ge
1$ with $1\le c_\nu < \cdots < c_{1}$ and $l_1,\ldots, l_\nu \in
\{1,\ldots, b-1\}$. First we consider the sum over all those $l$ for
which $1 \le \nu \le \alpha$. This part of the sum is bounded by
$$\sum_{\nu = 1}^\alpha (b-1)^{\nu} \sum_{c_1 = \nu}^\infty
\sum_{c_2 = \nu-1}^{c_1-1} \cdots \sum_{c_\nu = 1}^{c_{\nu-1}-1}
b^{-c_1 - \ldots - c_\nu}\le \sum_{\nu = 1}^\alpha (b-1)^{\nu}
(\sum_{c = 1}^\infty b^{-c})^\nu = \alpha.$$ If $\nu > \alpha$ we
have $q_{b,\alpha}(l) = q_{b,\alpha}(l')$ for $l = l_1 b^{c_1-1} +
\cdots + l_\nu b^{c_\nu-1}$ and where $l' = l'(l) = l_1 b^{c_1-1} +
\cdots + l_\alpha b^{c_\alpha-1}.$ Thus we only need to sum over all
$l'$ (i.e. natural numbers with exactly $\alpha$ digits) and for
given $l'$ multiplying it with the number of $l$ which yield the
same $l'$, which is $b^{c_\alpha-1} - 1$ (and which we bound in the
following by $b^{c_\alpha-1}$). We have
\begin{eqnarray*}
\lefteqn{ (b-1)^\alpha \sum_{c_1 = \alpha+1}^\infty \sum_{c_2 =
\alpha}^{c_1-1} \cdots \sum_{c_\alpha = 2}^{c_{\alpha-1}-1} b^{-c_1
- \cdots - c_\alpha} b^{c_\alpha - 1} } \\ & = & b^{-1}(b-1)^\alpha
\sum_{c_1 = \alpha+1}^\infty \sum_{c_2 = \alpha}^{c_1-1} \cdots
\sum_{c_{\alpha-1} = 3}^{c_{\alpha-2}-2} (c_{\alpha-1}-2) b^{-c_1 -
\cdots - c_{\alpha-1}} \\ & \le &  b^{-3}(b-1)^\alpha
(\sum_{c=1}^\infty b^{-c})^{\alpha-2} \sum_{c=1}^\infty c b^{-c} \\
& = & \frac{1}{b^2}.
\end{eqnarray*}
Thus we obtain $ \sum_{l=1}^\infty q_{b,\alpha}(l) \le \alpha + b^{-2}$.

Further we have $$ \sum_{\bsh \in \D_{b^m}}\sum_{\bsl \in \NN_0^s}
\sqrt{r_{b,\alpha}(\bsh + b^m \bsl)} = \sum_{\bsh \in \D_{b^m}}
\prod_{j=1}^s \sum_{l=0}^\infty \sqrt{r_{b,\alpha}(h_j + b^m l)},$$
where $\bsh = (h_1,\ldots, h_s)$. By using Lemma~\ref{lem_rbm} from
the Appendix and Lemma~\ref{lem_boundr} we obtain
$$\sum_{l=0}^{\infty} \sqrt{r_{b,\alpha}(b^m l)} = 1 +  b^{-\alpha
m} C_{b,\alpha} \sum_{l=1}^\infty q_{b,\alpha}(l) \le 1 + b^{-\alpha
m} C_{b,\alpha} (\alpha + b^{-2}).$$ Let now $0 < h_j < b^m$. From
Lemma~\ref{lem_boundr} we obtain  $$\sqrt{r_{b,\alpha}(h_j + b^m l)}
\le C_{b,\alpha} q_{b,\alpha}(h_j+ b^m l) \le C_{b,\alpha}
q_{b,\alpha}(h_j) q_{b,\alpha}(l).$$ From above we have
$\sum_{l=0}^\infty q_{b,\alpha}(l) \le 1 + \alpha + b^{-2}$ and
hence $$\sum_{l=0}^\infty \sqrt{r_{b,\alpha}(h_j + b^m l)} \le
q_{b,\alpha}(h_j)C_{b,\alpha} \sum_{l=0}^\infty q_{b,\alpha}(l) \le
C_{b,\alpha}(1 + \alpha + b^{-2}) q_{b,\alpha}(h_j).$$ Thus we
obtain
\begin{eqnarray*}
 \lefteqn{ \sum_{\bsh \in \D_{b^m}}\sum_{\bsl \in \NN_0^s} \sqrt{r_{b,\alpha}(\bsh + b^m \bsl)} } \\ &  = & \sum_{\emptyset \neq u \subseteq \{1,\ldots, s\}} \sum_{\bsh_u \in \D^\ast_{b^m,u}} \prod_{j \in u} \sum_{l = 0}^\infty \sqrt{r_{b,\alpha}(h_j + b^m l)} \prod_{j \not\in u} \sum_{l=0}^\infty \sqrt{r_{b,\alpha}(b^m l)}  \\  & \le &  \sum_{\emptyset \neq u \subseteq \{1,\ldots, s\}} (1+ b^{-\alpha m} C_{b,\alpha} (\alpha + b^{-2}))^{s-|u|}  C_{b,\alpha}^{|u|} (1+\alpha +b^{-2})^{|u|} \sum_{\bsh_u \in \D^\ast_{b^m,u}} \prod_{j \in u} q_{b,\alpha}(h_j),
\end{eqnarray*}
where $\bsh_u = (h_j)_{j\in u}$. The result follows.
\end{proof}

In \cite{shar} it was shown that the square worst-case error for numerical integration in the Korobov space can at best be of $\landau(N^{-2\alpha} (\log N)^{s-1})$, where $N$ is the number of quadrature points. Hence Lemma~\ref{lem_q} shows that it is enough to consider only $Q^\ast_{b,m,u,\alpha}(C_1,\ldots, C_s)$ in order to investigate the convergence rate of digitally shifted digital nets.

\section{$(t,\alpha,\beta,m,s)$-nets and $(t,\alpha,\beta,s)$-sequences}\label{sec_talpha}

The $t$ value of a $(t,m,s)$-net is a quality parameter for the distribution properties of the net. A low $t$ value yields well distributed point sets and it has been shown, see for example \cite{DP05, niesiam}, that a small $t$ value also guarantees a small worst-case error for integration in Sobolev spaces for which the partial first derivatives are square integrable.

In the following we will show how the definition of the $t$ value needs to be modified in order to obtain faster convergence rates for periodic Sobolev spaces for which the partial derivatives up to order $\alpha$ are square integrable. It is the aim of this definition to translate the problem of minimizing the worst-case error into an algebraical problem concerning the generating matrices. (This definition can therefore also be used in an computer search algorithm, where one could for example search for the polynomial lattice with the smallest $t(\alpha)$ value which in turn yields a small worst-case error for integration of periodic functions.)

For natural numbers $\alpha \ge 1$, Lemma~\ref{lem_boundr} suggests to define the following metric $\mu_{b,\alpha}(\bsk,\bsl) = \mu_{b,\alpha}(\bsk\ominus \bsl)$ on $\NN_0^s$ which is an extension of the metric introduced in \cite{nie86}, see also \cite{rt} (for $\alpha = 1$ we basically obtain the metric in \cite{nie86, rt}). Here $\mu_{b,\alpha}(0) = 0 $ and for $k\in \NN$ with $k = \kappa_{\nu} b^{a_{\nu}-1} + \cdots + \kappa_{1} b^{a_1-1}$ where $1\le a_{\nu} < \cdots < a_1$ and $\kappa_{i} \in \{1,\ldots, b-1\}$ let $\mu_{b,\alpha}(k) = a_{1} + \cdots + a_{\min(\alpha,\nu)}$. For a $\bsk \in \NN_0^s$ with $\bsk = (k_1,\ldots, k_s)$ let $\mu_{b,\alpha}(\bsk) = \mu_{b,\alpha}(k_1) + \cdots + \mu_{b,\alpha}(k_s)$. Then we have $q_{b,\alpha}(\bsk) = b^{-\mu_{b,\alpha}(\bsk)}$. Hence in order to obtain a small worst-case error in the Korobov space $\HH_\alpha$, we need digital nets for which $\min\{\mu_{b,\alpha}(\bsk): \bsk \in \D\}$ is large. We can translate this property into a linear independence property of the row vectors of the generating matrices $C_1,\ldots, C_s$. We have the following definition.

\begin{definition}\rm\label{def_net}
Let $m,\alpha \ge 1$ be natural numbers, let $0 <\beta \le \alpha$
be a real number and let $0 \le t \le \beta m$  be  a natural
number. Let $R_b$ be a ring with $b$ elements and let $C_1,\ldots,
C_s \in R_b^{m \times m}$ with $C_j = (c_{j,1}, \ldots, c_{j,m})^T$.
If for all $1 \le i_{j,\nu_j} < \cdots < i_{j,1} \le m$, where $0
\le \nu_j \le m$ for all $j = 1,\ldots, s$, with $$i_{1,1} + \cdots
+ i_{1,\min(\nu_1,\alpha)} + \cdots + i_{s,1} + \cdots +
i_{s,\min(\nu_s,\alpha)} \le \beta m - t$$ the vectors
$$c_{1,i_{1,\nu_1}}, \ldots, c_{1,i_{1,1}}, \ldots,
c_{s,i_{s,\nu_s}}, \ldots, c_{s,i_{s,1}}$$ are linearly independent
over $R_b$ then the digital net which has generating matrices
$C_1,\ldots, C_s$ is called a digital $(t,\alpha,\beta,m,s)$-net
over $R_b$. Further we call a digital $(t,\alpha,\alpha,m,s)$-net
over $R_b$ a digital $(t,\alpha,m,s)$-net over $R_b$.

If $t$ is the smallest non-negative integer such that the digital
net generated by $C_1,\ldots, C_s$ is a digital
$(t,\alpha,\beta,m,s)$-net, then we call the digital net a strict
digital $(t,\alpha,\beta, m,s)$-net or a strict digital
$(t,\alpha,m,s)$-net if $\alpha = \beta$.
\end{definition}

A concrete example of a digital $(t,\alpha,\beta,m,s)$-net, where we
also calculate the exact $t$-value by hand, is given in
Section~\ref{sec_example}.

\begin{remark}\rm
Using duality theory (see \cite{np}) it follows that for every digital $(t,\alpha,\beta,m,s)$-net we have $\min_{\bsk \in \D} \mu_{b,\alpha}(\bsk) > \beta m - t$ and for a strict digital $(t,\alpha,\beta,m,s)$-net we have $\min_{\bsk \in \D} \mu_{b,\alpha}(\bsk) = \beta m - t + 1$. Hence digital $(t,\alpha,\beta,m,s)$-nets with high quality have a large value of $\beta m - t$.
\end{remark}

\begin{definition}\rm\label{def_seq}
Let $\alpha \ge 1$ and $t \ge 0$ be integers and let $0 <\beta \le \alpha$ be a real number. Let $R_b$ be a ring with $b$ elements and let $C_1,\ldots, C_s \in R_b^{\infty \times \infty}$ with $C_j = (c_{j,1}, c_{j,2}, \ldots)^T$. Further let $C_{j,m}$ denote the left upper $m \times m$ submatrix of  $C_j$. If for all $m > t/\beta$ the matrices $C_{1,m},\ldots, C_{s,m}$ generate a digital $(t,\alpha,\beta,m,s)$-net then the digital sequence with generating matrices $C_1,\ldots, C_s$ is called a digital $(t,\alpha,\beta,s)$-sequence over $R_b$. Further we call a digital $(t,\alpha,\alpha,s)$-sequence over $R_b$ a digital $(t,\alpha,s)$-sequence over $R_b$.

If $t$ is the smallest non-negative integer such that the digital sequence generated by $C_1,\ldots, C_s$ is a digital $(t,\alpha,\beta,s)$-sequence, then we call the digital sequence a strict digital $(t,\alpha,\beta, s)$-sequence or a strict digital $(t,\alpha,s)$-sequence if $\alpha = \beta$.
\end{definition}

\begin{remark}\rm
Note that the definition of a digital $(t, 1,m,s)$-net coincides with the definition of a digital $(t,m,s)$-net and the definition of a digital $(t, 1,s)$-sequence coincides with the definition of a digital $(t,s)$-sequence as defined by Niederreiter~\cite{niesiam}. Further note that the $t$-value depends on $\alpha$ and $\beta$, i.e., $t = t(\alpha,\beta)$ or $t = t(\alpha)$ if $\alpha = \beta$.
\end{remark}

In the following theorem we establish some propagation rules.
\begin{theorem}\label{th_prop}
Let $P$ be a digital $(t,\alpha,\beta,m,s)$-net over a ring $R_b$ and let  $S$ be a digital $(t,\alpha,\beta,s)$-sequence over a ring $R_b$. Then we have:
\begin{enumerate}
\item[(i)] $P$ is a digital $(t',\alpha,\beta',m,s)$-net for all $1\le \beta' \le \beta$ and  all $t \le t' \le \beta' m$ and $S$ is a digital $(t',\alpha,\beta',s)$-sequence for all $1 \le \beta' \le \beta$ and all $t \le t'$.
\item[(ii)] $P$ is a digital $(t',\alpha',\beta',m,s)$-net for all $1\le \alpha' \le m$ and $S$ is a digital $(t',\alpha',\beta',s)$-sequence for all $\alpha' \ge 1$, where  $\beta' = \beta \min(\alpha,\alpha')/\alpha$ and $t' =  \lceil t \min(\alpha,\alpha')/\alpha\rceil$.
\item[(iii)] Any digital $(t,\alpha,m,s)$-net is a digital $(\lceil t \alpha'/\alpha\rceil ,\alpha',m,s)$-net for all $1\le \alpha' \le \alpha$ and every digital $(t,\alpha,s)$-sequence is a digital $(\lceil t\alpha'/\alpha \rceil,\alpha',s)$-sequence for all $1\le \alpha' \le \alpha$.
\end{enumerate}
\end{theorem}

\begin{proof}
Note that it follows from Definition~\ref{def_seq} that we need to prove the result only for digital nets.

The first part follows trivially. To prove the second part choose an $\alpha'$ such that $\alpha' \ge 1$. Then choose arbitrary $1 \le i_{j,\nu_j} < \cdots < i_{j,1} \le m$ with $0 \le \nu_j \le m$ such that $$i_{1,1} + \cdots + i_{1,\min(\nu_1,\alpha')} + \cdots + i_{s,1} + \cdots + i_{s,\min(\nu_s,\alpha')} \le m\beta \frac{\min(\alpha,\alpha')}{\alpha} - \left\lceil t \frac{\min(\alpha,\alpha')}{\alpha} \right\rceil.$$ We need to show that the vectors $$c_{1,i_{1,\nu_1}}, \ldots, c_{1,i_{1,1}}, \ldots, c_{s,i_{s,\nu_s}}, \ldots, c_{s,i_{s,1}}$$ are linearly independent over $R_b$. This is certainly the case as long as $$i_{1,1} + \cdots + i_{1,\min(\nu_1,\alpha)} + \cdots + i_{s,1} + \cdots + i_{s,\min(\nu_s,\alpha)} \le \beta m -  t.$$ Indeed we have
\begin{eqnarray*}
\lefteqn{ i_{1,1} + \cdots + i_{1,\min(\nu_1,\alpha)} + \cdots + i_{s,1} + \cdots + i_{s,\min(\nu_s,\alpha)} } \\ &\le & \frac{\alpha}{\min(\alpha,\alpha')} (i_{1,1} + \cdots + i_{1,\min(\nu_1,\alpha')} + \cdots + i_{s,1} + \cdots + i_{s,\min(\nu_s,\alpha')}) \\ &\le & m \beta - \frac{\alpha}{\min(\alpha,\alpha')}\left\lceil t\frac{\min(\alpha,\alpha')}{\alpha} \right\rceil \\ & \le & m\beta - t,
\end{eqnarray*}
and hence the second part follows. The third part is just a special case of the second part.
\end{proof}

\begin{remark}\rm
Note by choosing $\alpha'= 1$ in part $(iii)$ of Theorem~\ref{th_prop} it follows that digital  $(t,\alpha,m,s)$-nets and digital  $(t,\alpha,s)$-sequences are also well distributed point sets if the value of $t$ is small, see \cite{niesiam}.
\end{remark}

\section{Explicit constructions of digital $(t,\alpha,\beta,m,s)$-nets and digital $(t,\alpha,\beta,s)$-sequences}\label{sec_talphacons}

In this section we show how suitable digital $(t,\alpha,\beta,m,s)$-nets and digital $(t,\alpha,\beta,s)$-sequences can be constructed.

Let $d \ge 1$ and let $C_1,\ldots, C_{sd}$ be the generating matrices of a digital $(t,m,sd)$-net. Note that many explicit examples of such generating matrices are known, see for example \cite{faure, niesiam, NX, sob67} and the references therein. For the construction of a $(t,\alpha,\beta,m,s)$-net any of the above mentioned explicit constructions can be used, but as will be shown below the quality of the $(t,\alpha,\beta,m,s)$-net obtained depends on the quality of the underlying digital $(t,m,sd)$-net on which our construction is based on.

Let $C_j = (c_{j,1},\ldots, c_{j,m})^T$ for $j = 1,\ldots, sd$, i.e., $c_{j,l}$ are the row vectors of $C_j$. Now let the matrix $C^{(d)}_{j}$ be made of the first rows of the matrices $C_{(j-1)d + 1},\ldots, C_{jd}$, then the second rows of $C_{(j-1)d+1},\ldots, C_{jd}$ and so on till $C^{(d)}_{j}$ is an $m \times m$ matrix, i.e., $C^{(d)}_j = (c^{(d)}_{j,1},\ldots, c^{(d)}_{j,m})^T$ where $c^{(d)}_{j,l} = c_{u,v}$ with $l = (v-j)d + u$, $1\le v \le m$ and $(j-1)d < u \le jd$ for $l = 1,\ldots, m$ and $j = 1,\ldots, s$. In the following we will show that the matrices $C^{(d)}_{1},\ldots, C^{(d)}_{s}$ are the generating matrices of a digital $(t,\alpha,\min(\alpha,d),m,s)$-net.

\begin{theorem}\label{th_talphabeta}
Let $d \ge 1$ be a natural number and let $C_{1},\ldots, C_{sd}$ be
the generating matrices of a digital $(t',m,sd)$-net over some ring
$R_b$ with $b$ elements. Let $C^{(d)}_{1},\ldots, C^{(d)}_{s}$ be
defined as above. Then for any $\alpha \ge 1$ the matrices
$C^{(d)}_{1},\ldots, C^{(d)}_{s}$ are  generating matrices of a
digital $(t,\alpha,\min(\alpha,d),m,s)$-net over $R_b$ with $$t =
\min(\alpha,d)\;t' + \left\lceil \frac{s(d-1)
\min(\alpha,d)}{2}\right\rceil.$$
\end{theorem}

\begin{proof}
Let $C^{(d)}_j = (c_{j,1}^{(d)},\ldots, c_{j,m}^{(d)})^T$ for $j = 1,\ldots, s$ and further let the integers $i_{1,1},\ldots, i_{1,\nu_1},\ldots, i_{s,1},\ldots, i_{s,\nu_s}$ be such that $1 \le i_{j,\nu_j} < \cdots < i_{j,1}\le m$ and $$i_{1,1} + \cdots + i_{1,\min(\nu_1,\alpha)} + \cdots + i_{s,1} + \cdots +  i_{s,\min(\nu_s,\alpha)} \le \min(\alpha,d) m - t.$$ We need to show that the vectors  $$c^{(d)}_{1,i_{1,1}},  \ldots, c^{(d)}_{1,i_{1,\nu_1}}, \ldots, c^{(d)}_{s,i_{s,1}},  \ldots, c^{(d)}_{s,i_{s,\nu_s}}$$ are linearly independent over $R_b$.  For $j = 1,\ldots, s$ let  $U_j = \{c^{(d)}_{j,i_{j,\nu_j}},\ldots, c^{(d)}_{j,i_{j,1}}\}$. The vectors in the set $U_j$ stem from the matrices $C_{(j-1)d +1},\ldots, C_{jd}$. For $j = 1,\ldots, s$ and  $d_j = (j-1)d + 1,\ldots, jd$ let $e_{d_j}$ denote the largest index such that $(e_{d_j}-j)d + d_j \in \{i_{j,\nu_j},\ldots, i_{j,1}\}$ and if for some $d_j$ there is no such $e_{d_j}$ we set $e_{d_j}=0$ (basically this means $e_{d_j}$ is the largest integer such that $c_{d_j,e_{d_j}} \in U_j$).

Let $d \le \alpha$, then we have $d ((e_{(j-1)d + 1}-1)_+ + \cdots +
(e_{jd}-1)_+) + \sum_{l=1}^{L_j} l \le i_{j,1} + \cdots +
i_{j,\min(\nu_j,d)}$ where $(x)_+ = \max(x,0)$ and $L_j = |\{(j-1)d
+ 1 \le d_j \le jd: e_{d_j} > 0\}|$. Hence we have
\begin{eqnarray}\label{ungl}
\lefteqn{d ((e_{(j-1)d + 1}-1)_+ + \cdots + (e_{jd}-1)_+) +
\sum_{l=1}^{L_j} l} \nonumber \\ &= & d(e_{(j-1)d+1} + \cdots +
e_{jd}) - L_j d + L_j(L_j+1)/2 \nonumber \\ & \ge &d(e_{(j-1)d+1} +
\cdots + e_{jd}) - \frac{d(d-1)}{2}.
\end{eqnarray}
Thus it follows that $$d(e_1 + \cdots + e_{sd}) \le \sum_{j=1}^s (i_{j,1} + \cdots + i_{j,\min(\nu_j,\alpha)}) + s\frac{d(d-1)}{2} \le d m - t + s\frac{d(d-1)}{2}$$ and therefore $$e_1 + \cdots + e_{sd} \le m - \frac{t}{d} + s\frac{d-1}{2}\le m - t'.$$ Thus it follows from the $(t',m,sd)$-net property of the digital net generated by $C_1,\ldots, C_{sd}$ that the vectors $c^{(d)}_{1,i_{1,1}},  \ldots, c^{(d)}_{1,i_{1,\nu_1}}, \ldots, c^{(d)}_{s,i_{s,1}},  \ldots, c^{(d)}_{s,i_{s,\nu_s}}$ are linearly independent.

Let now $d > \alpha$. Then we have $d((e_{(j-1)d+1}-1)_+ + \cdots +
(e_{jd}-1)_+)) + \sum_{l=1}^{L_j} l \le i_{j,1} + \cdots +
i_{j,\min(\nu_j,\alpha)} + (d-\alpha)i_{j,\min(\nu_j,\alpha)}$,
where again $L_j = |\{(j-1)d + 1 \le d_j \le jd: e_{d_j} > 0\}|$.
Hence we can use inequality (\ref{ungl}) again.
Note that $i_{1,\min(\nu_1,\alpha)} + \cdots +
i_{s,\min(\nu_s,\alpha)} \le m - t/\alpha$ and hence we have
$$\sum_{j=1}^s (i_{j,1} + \cdots + i_{j,\min(\nu_j,\alpha)} +
(d-\alpha)i_{j,\min(\nu_j,\alpha)}) \le \alpha m - t + (d - \alpha)
(m-t/\alpha) = d m - d t/\alpha.$$ Thus it follows that
\begin{eqnarray*}
d(e_1 + \cdots + e_{sd})& \le & \sum_{j=1}^s (i_{j,1} + \cdots + i_{j,\min(\nu_j,\alpha)} + (d - \alpha) i_{j,\min(\nu_j,\alpha)}) + s\frac{d(d-1)}{2} \\ & \le & d m - \frac{d t}{\alpha} + s\frac{d(d-1)}{2}
\end{eqnarray*}
and therefore $$e_1 + \cdots + e_{sd} \le m - \frac{t}{\alpha} + s \frac{d-1}{2} \le m - t'.$$ Thus it follows from the $(t',m,sd)$-net property of the digital net generated by $C_1,\ldots, C_{sd}$ that the vectors  $c^{(d)}_{1,i_{1,1}},  \ldots, c^{(d)}_{1,i_{1,\nu_1}}, \ldots, c^{(d)}_{s,i_{s,1}},  \ldots, c^{(d)}_{s,i_{s,\nu_s}}$  are linearly independent and hence the result follows.
\end{proof}

In Section~\ref{sec_example} we use this construction method to
construct a digital $(3,2,4,2)$-net over $\integer_2$.

Note that the construction and Theorem~\ref{th_talphabeta} can
easily be extended to $(t,\alpha,\beta,s)$-sequences. Indeed, let $d
\ge 1$ and let $C_1,\ldots, C_{sd}$ be the generating matrices of a
digital $(t,sd)$-sequence. Again many explicit generating matrices
are known, see for example \cite{faure, niesiam, NX, sob67}. Let
$C_j = (c_{j,1},c_{j,2},\ldots)^T$ for $j = 1,\ldots,sd$, i.e.,
$c_{j,l}$ are the row vectors of $C_j$. Now let the matrix
$C^{(d)}_{j}$ be made of the first rows of the matrices $C_{(j-1)d +
1},\ldots, C_{jd}$, then the second rows of $C_{(j-1)d+1},\ldots,
C_{jd}$ and so on, i.e., $$C^{(d)}_{j} = (c_{(j-1)d+1,1},\ldots,
c_{jd,1},c_{(j-1)d+1,2},\ldots,c_{jd,2},\ldots)^T.$$ The following
theorem states that the matrices $C^{(d)}_{1},\ldots, C^{(d)}_{s}$
are the generating matrices of a digital
$(t,\alpha,\min(\alpha,d),s)$-sequence.

\begin{theorem}\label{th_talphabetaseq}
Let $d \ge 1$ be a natural number and let $C_{1},\ldots, C_{sd}$ be
the generating matrices of a digital $(t',sd)$-sequence over some
ring $R_b$ with $b$ elements. Let $C^{(d)}_{1},\ldots, C^{(d)}_{s}$
be defined as above. Then for any $\alpha \ge 1$ the matrices
$C^{(d)}_{1},\ldots, C^{(d)}_{s}$ are generating matrices of a
digital $(t,\alpha,\min(\alpha,d),s)$-sequence over $R_b$ with $$t =
\min(\alpha,d)\;t' + \left\lceil \frac{s(d-1)
\min(\alpha,d)}{2}\right\rceil.$$
\end{theorem}

The last result shows that $(t,\alpha,\beta,m,s)$-nets indeed exist for any $0 < \beta \le \alpha$ and for $m$ arbitrarily large. We have even shown that digital $(t,\alpha,\beta,m,s)$-nets exist which are extensible in $m$ and $s$. This can be achieved by using an underlying $(t',sd)$-sequence which is itself extensible in $m$ and $s$. If the $t'$ value of the original $(t',m,s)$-net or $(t',s)$-sequence is known explicitly then we also know the $t$ value of the digital $(t,\alpha,\beta,m,s)$-net or $(t,\alpha,\beta,s)$-sequence. Furthermore it has also been shown how such digital nets can be constructed in practise.

In the following we investigate for which values of $t,\alpha,s,b$ digital $(t,\alpha,s)$-sequences over $\integer_b$ exist. We need some further notation (see also \cite{NX2}, Definition~8.2.15).
\begin{definition}\rm
For given integers $s,\alpha \ge 1$ and prime number $b$ let $d_b(s,\alpha)$ be the smallest value of $t$ such that a $(t,\alpha,s)$-sequence over $\integer_b$ exists.
\end{definition}
We have the following bound on $d_b(s,\alpha)$.

\begin{corollary}\label{cor_tbound}
Let $s,\alpha \ge 1$ be integers and $b$ be a prime number. Then we have
\begin{eqnarray*}
\lefteqn{ \alpha \left(\frac{s}{b} -1 - \log_b \frac{(b-1)s+b+1}{2}\right) +1 } \qquad\qquad \\ & \le & d_b(s,\alpha) \;\; \le \;\; \alpha (s-1) \frac{3b-1}{b-1} - \alpha \frac{(2b+4)\sqrt{s-1}}{\sqrt{b^2-1}} + 2 \alpha + s \frac{\alpha (\alpha-1)}{2}.
\end{eqnarray*}
\end{corollary}

\begin{proof}
The lower bound follows from part $(iii)$ of Theorem~\ref{th_prop} by choosing $\alpha' = 1$ and using a lower bound on the $t$-value for $(t,s)$-sequences (see \cite{NX}).  The upper bound follows from Theorem~\ref{th_talphabetaseq} by choosing $d = \alpha$ and using Theorem~8.4.4 of \cite{NX2}.
\end{proof}

\section{A bound on the worst-case error in $\HH_\alpha$ for digital $(t,\alpha,\beta,m,s)$-nets and digital $(t,\alpha,\beta,s)$-sequences}\label{sec_bound}

In this section we prove an upper bound on the worst-case error for integration in the Korobov space $\HH_\alpha$ using digital $(t,\alpha,\beta,m,s)$-nets and $(t,\alpha,\beta,s)$-sequences.

\begin{lemma}\label{lem_Q}
Let $\alpha \ge 2$ be a natural number, let $b$ be prime and let $C_1,\ldots, C_s \in \integer_b^{m \times m}$ be the generating matrices of a digital $(t,\alpha,\beta,m,s)$-net over $\integer_b$ with $m > t/\beta$. Then we have
\begin{eqnarray*}
Q^\ast_{b,m,u,\alpha}(C_1,\ldots, C_s) \le 2b^{|u|\alpha} b^{-\beta m + t} (\beta m + 2)^{|u|\alpha - 1},
\end{eqnarray*}
where $Q^\ast_{b,m,u,\alpha}$ is defined in Lemma~\ref{lem_q}.
\end{lemma}

\begin{proof}
We obtain a bound on $Q^\ast_{b,m,\{1,\ldots, s\},\alpha}$, for all
other subsets $u$ the bound can be obtained using the same
arguments.

We first partition the set $\D^\ast_{b^m,\{1,\ldots, s\}}$ into parts where the highest digits of $k_j$ are prescribed and we count the number of solutions of $C_1^T\vec{k}_1 + \cdots + C_s^T \vec{k}_s = \vec{0}$.
For $j = 1,\ldots, s$ let now $i_{j,\alpha} < \cdots < i_{j,1} \le m$ with $i_{j,1} \ge 1$.  Note that we now allow $i_{j,l} < 1$, in which case the contributions of those $i_{j,l}$ are to be ignored. This notation is adopted in order to avoid considering many special cases. Now we define
\begin{eqnarray*}
\lefteqn{ \D^\ast_{b^m,\{1,\ldots, s\}}(i_{1,1},\ldots, i_{1,\alpha},\ldots, i_{s,1},\ldots, i_{s,\alpha}) } \\ &=& \{\bsk \in \D^\ast_{b^m,\{1,\ldots, s\}}: k_j = \lfloor \kappa_{j,1} b^{i_{j,1}-1} + \cdots + \kappa_{j,\alpha}b^{i_{j,\alpha}-1} + l_j\rfloor  \mbox{ with } 0 \le l_j < b^{i_{j,\alpha}-1} \\ && \mbox{and } 1 \le \kappa_{j,l} < b \mbox{ for } j = 1,\ldots, s\},
\end{eqnarray*}
where $\lfloor\cdot \rfloor$ just means that the contributions of $i_{j,l} < 1$ are to be ignored. Then we have
\begin{eqnarray}\label{sum_Qstar}
\lefteqn{ Q^\ast_{b,m,\{1,\ldots, s\},\alpha}(C_1,\ldots, C_s) } \nonumber \\ & = & \sum_{i_{1,1}=1}^m \cdots \sum_{i_{1,\alpha}= 1}^{i_{1,\alpha-1}-1} \cdots \sum_{i_{s,1}=1}^m \cdots \sum_{i_{s,\alpha}=1}^{i_{s,\alpha-1}-1} \frac{| \D^\ast_{b^m,\{1,\ldots, s\}}(i_{1,1},\ldots, i_{1,\alpha},\ldots, i_{s,1},\ldots, i_{s,\alpha})|}{b^{i_{1,1}+ \cdots + i_{1,\alpha} + \cdots + i_{s,1} + \cdots  + i_{s,\alpha}}}.
\end{eqnarray}
Some of the sums above can be empty in which case we just set the corresponding summation index $i_{j,l}=0$.

Note that by the $(t,\alpha,\beta,m,s)$-net property we have $$|\D^\ast_{b^m,\{1,\ldots, s\}}(i_{1,1},\ldots, i_{1,\alpha},\ldots, i_{s,1},\ldots, i_{s,\alpha})| = 0$$ as long as $i_{1,1} + \cdots + i_{1,\alpha} + \cdots + i_{s,1} + \cdots + i_{s,\alpha} \le \beta m - t$. Hence let now $0\le i_{1,1}, \ldots, i_{s,\alpha} \le m$ be given such that $i_{1,1},\ldots, i_{s,1}\ge 1$, $i_{j,\alpha}< \cdots < i_{j,1} \le m$ for $j = 1,\ldots, s$ and where if $i_{j,l} < 1$ we set $i_{j,l}=0$  and  $i_{1,1} + \cdots + i_{1,\alpha} + \cdots + i_{s,1} + \cdots  + i_{s,\alpha} > \beta m - t$. We now need to estimate $|\D^\ast_{b^m,\{1,\ldots, s\}}(i_{1,1},\ldots, i_{1,\alpha},\ldots, i_{s,1},\ldots, i_{s,\alpha})|$, that is we need to count the number of $\bsk \in \D^\ast_{b^m,\{1,\ldots, s\}}$ with $k_j = \lfloor \kappa_{j,1} b^{i_{j,1}-1} + \cdots + \kappa_{j,\alpha}b^{i_{j,\alpha}-1} + l_j\rfloor$ such that $C_1^T \vec{k}_1 + \cdots + C_s^T\vec{k}_s = \vec{0}$.

There are at most $(b-1)^{\alpha s}$ choices for $\kappa_{1,1},\ldots, \kappa_{s,\alpha}$ (we write at most because if $i_{j,l} < 1$ then the corresponding $\kappa_{j,l}$ does not have any effect and therefore need not to be included). Let now $1\le \kappa_{1,1},\ldots, \kappa_{s,\alpha} < b$ be given and define $$\vec{g} = \kappa_{1,1} c_{1,i_{1,1}}^T + \cdots + \kappa_{1,\alpha} c_{1,i_{1,\alpha}}^T + \cdots + \kappa_{s,1} c_{s,i_{s,1}}^T + \cdots + \kappa_{s,\alpha} c_{s,i_{s,\alpha}}^T,$$ where we set $c^T_{j,l}=0$ if $l < 1$. Further let $$B = (c_{1,1}^T,\ldots, c_{1,i_{1,\alpha}-1}^T,\ldots, c_{s,1}^T,\ldots, c_{s,i_{s,\alpha}-1}^T).$$ Now the task is to count the number of solutions $\vec{l}$ of $B \vec{l} = \vec{g}$. As long as the columns of $B$ are linearly independent the number of solutions can at most be $1$. By the $(t,\alpha,\beta,m,s)$-net property this is certainly the case if (we write $(x)_+ = \max(x,0)$)
\begin{eqnarray*}
\lefteqn{ (i_{1,\alpha}-1)_+ + \cdots + (i_{1,\alpha}-\alpha)_+ + \cdots + (i_{s,\alpha}-1)_+ + \cdots + (i_{s,\alpha}-\alpha)_+ } \qquad\qquad\qquad\qquad\qquad\qquad\qquad\qquad\qquad\qquad\qquad\qquad \\  &\le & \alpha (i_{1,\alpha} + \cdots + i_{s,\alpha})  \\ & \le & \beta m - t,
\end{eqnarray*}
that is, as long as $$i_{1,\alpha} + \cdots + i_{s,\alpha} \le \frac{\beta m - t}{\alpha}.$$

Let now $i_{1,\alpha} + \cdots + i_{s,\alpha} > \frac{\beta m - t}{\alpha}$. Then by considering the rank of the matrix $B$ and the dimension of the space of solutions of $B\vec{l} = \vec{0}$ it follows the number of solutions of $B\vec{l} = \vec{g}$ is smaller or equal to $b^{i_{1,\alpha} + \cdots + i_{s,\alpha}-\lfloor(\beta m - t)/\alpha\rfloor}$. Thus we have
\begin{eqnarray*}
\lefteqn{ |\D^\ast_{b^m,\{1,\ldots, s\}}(i_{1,1},\ldots, i_{1,\alpha},\ldots, i_{s,1},\ldots, i_{s,\alpha})| }\qquad\qquad\qquad\qquad\qquad\qquad\qquad\qquad\qquad\qquad\qquad\qquad\qquad\qquad\qquad\qquad \\ \le \left\{\begin{array}{ll} 0 & \mbox{if } \sum_{j=1}^s \sum_{l = 1}^\alpha i_{j,l}  \le \beta m -t, \\ (b-1)^{\alpha s} & \mbox{if } \sum_{j=1}^s \sum_{l=1}^\alpha i_{j,l}  > \beta m -t \\ & \mbox{ and } \sum_{j=1}^s i_{j,\alpha} \le \frac{\beta m -t}{\alpha}, \\ (b-1)^{\alpha s} b^{i_{1,\alpha} + \cdots + i_{s,\alpha} - \lfloor (\beta m -t)/\alpha\rfloor} & \mbox{if } \sum_{j=1}^s\sum_{l=1}^\alpha i_{j,l} > \beta m -t \\ & \mbox{ and } \sum_{j=1}^s i_{j,\alpha} > \frac{\beta m -t}{\alpha}.  \end{array} \right.
\end{eqnarray*}

We estimate the sum (\ref{sum_Qstar}) now. Let $S_1$ be the sum in
(\ref{sum_Qstar}) where $ i_{1,1} + \cdots + i_{s,\alpha} > \beta m
-t$ and $i_{1,\alpha} + \cdots + i_{s,\alpha} \le \frac{\beta m
-t}{\alpha}$. For an $l > \beta m - t$ let $A_1(l)$ denote the
number of admissible choices of $i_{1,1},\ldots,i_{s,\alpha}$ such
that $l = i_{1,1} + \cdots + i_{s,\alpha}$. Then we have $$S_1 =
(b-1)^{\alpha s} \sum_{l= \beta m -t +1}^{\alpha s m}
\frac{A_1(l)}{b^{l}}.$$ We have $A_1(l) \le {l+s\alpha -1 \choose
s\alpha - 1}$ and hence we obtain $$S_1 \le (b-1)^{s\alpha} \sum_{l
= \beta m - t+1}^{\infty} {l+s\alpha - 1 \choose s \alpha -1}
\frac{1}{b^{l}} \le b^{s\alpha}b^{-\beta m + t - 1}{\beta m - t +
s\alpha \choose s \alpha - 1},$$ where the last inequality follows
from a result by Matou\~sek~\cite[Lemma~2.18]{matou}, see also
\cite[Lemma~6]{DP05}.

Let $S_2$ be the part of (\ref{sum_Qstar}) for which $ i_{1,1} + \cdots + i_{s,\alpha} > \beta m -t$ and $i_{1,\alpha} + \cdots + i_{s,\alpha} > \frac{\beta m -t}{\alpha}$, i.e., we have
\begin{eqnarray}\label{sum_S2}
S_2 & =& (b-1)^{s\alpha} \sum_{i_{1,1}=1}^m \cdots \sum_{i_{1,\alpha}= 1}^{i_{1,\alpha-1}-1} \cdots \sum_{i_{s,1}=1}^m \cdots \sum_{i_{s,\alpha}=1}^{i_{s,\alpha-1}-1} \frac{b^{- \lfloor (\beta m - t)/\alpha\rfloor}}{b^{i_{1,1}+ \cdots + i_{1,\alpha-1} + \cdots + i_{s,1} + \cdots  + i_{s,\alpha-1}}} \nonumber \\ & \le & \frac{m^s (b-1)^{s\alpha}}{b^{\lfloor( \beta m - t)/\alpha \rfloor}} \sum_{i_{1,1}=1}^m \cdots \sum_{i_{1,\alpha-1}= 1}^{i_{1,\alpha-2}-1} \cdots \sum_{i_{s,1}=1}^m \cdots \sum_{i_{s,\alpha-1}=1}^{i_{s,\alpha-2}-1} \frac{1}{b^{i_{1,1}+ \cdots + i_{1,\alpha-1} + \cdots + i_{s,1} + \cdots  + i_{s,\alpha-1}}}
\end{eqnarray}
where in the first line above we have the additional conditions $ i_{1,1} + \cdots + i_{s,\alpha} > \beta m -t$ and $i_{1,\alpha} + \cdots + i_{s,\alpha} > \frac{\beta m -t}{\alpha}$. From the last inequality and $i_{1,\alpha-l} + \cdots + i_{s,\alpha - l} > i_{1,\alpha} + \cdots + i_{s,\alpha}$ for $l=1,\ldots, \alpha - 1$ it follows that $i_{1,1} + \cdots + i_{1,\alpha-1} + \cdots + i_{s,1} + \cdots + i_{s,\alpha - 1} \ge \lfloor (\beta m - t) (1 - \alpha^{-1})\rfloor +1$. Let $A_{2}(l)$ denote the number of admissible choices of $i_{1,1},\ldots, i_{1,\alpha-1},\ldots, i_{s,1},\ldots,i_{s,\alpha-1}$ such that $l = i_{1,1} + \cdots + i_{1,\alpha-1} + \cdots + i_{s,1} + \cdots +  i_{s,\alpha-1}$. Note that we have $A_{2}(l) \le {l + s (\alpha - 1) -1 \choose s(\alpha - 1)-1}$. Then we have
\begin{eqnarray*}
S_{2} & \le & \frac{m^s (b-1)^{s\alpha}}{b^{\lfloor( \beta m - t)/\alpha \rfloor}} \sum_{l = \lfloor (\beta m - t)(1 - \alpha^{-1})\rfloor + 1}^{\infty} {l + s(\alpha - 1)-1\choose s(\alpha - 1) - 1} \frac{1}{b^{l}} \\ & \le & \frac{m^s (b-1)^{s\alpha}}{b^{\lfloor( \beta m - t)/\alpha \rfloor}} \frac{b^{ \lceil (\beta m - t)/\alpha\rceil}}{(1-b^{-1})^{s(\alpha-1)}  b^{\beta m - t + 1}} { \lfloor (\beta m - t)(1 - \alpha^{-1})\rfloor + s (\alpha -1) \choose s(\alpha - 1) - 1},
\end{eqnarray*}
where the last inequality follows again from a result by
Matou\~sek~\cite[Lemma~2.18]{matou}, see also \cite[Lemma~6]{DP05}.
Hence we have
\begin{equation*}
S_{2} \le m^s b^{s\alpha} b^{-\beta m + t}{ \lfloor (\beta m - t)(1 - \alpha^{-1})\rfloor + s (\alpha -1) \choose s(\alpha - 1) - 1}.
\end{equation*}

Note that we have $Q^\ast_{b,m,\alpha,\{1,\ldots, s\}}(C_1,\ldots, C_s) = S_1 + S_2$. Let $a \ge 1$ and $b\ge 0$ be integers then we have $${a+b \choose b} = \prod_{i=1}^b \left(1 + \frac{a}{i}\right) \le (1+a)^b.$$ Therefore we obtain $S_1 \le b^{s\alpha}b^{-\beta m + t-1} (\beta m - t + 2)^{s \alpha - 1}$ and $S_{2} \le b^{s\alpha} b^{-\beta m + t} m^s (\beta m - t+ 2)^{s(\alpha - 1)-1}$. Thus we have $$Q^\ast_{b,m,\alpha,\{1,\ldots, s\}}(C_1,\ldots, C_s) \le 2b^{s\alpha} b^{-\beta m + t} (\beta m + 2)^{s\alpha - 1},$$
from which the result follows.
\end{proof}

The following theorem is an immediate consequence of Lemma~\ref{lem_q} and Lemma~\ref{lem_Q}.

\begin{theorem}\label{th_wce}
Let $b$ be prime, $\alpha \ge 2$ be a natural number and let $C_1,\ldots, C_s \in \integer_b^{m \times m}$ be the generating matrices of a digital $(t,\alpha,\beta,m,s)$-net over $\integer_b$ with $m > t/\beta$. Then the  worst-case error in the Korobov space $\HH_\alpha$ is bounded by
\begin{eqnarray*}
 e_{b,m,\alpha}(C_1,\ldots, C_s)  & \le & \frac{2\left(1 + b^{-\alpha m} C_{b,\alpha}(\alpha + b^{-2}) + C_{b,\alpha} (1 + \alpha + b^{-2}) (\beta m + 2)^\alpha \right)^s}{b^{\beta m - t} (\beta m + 2)} \\ && + (1+ b^{-\alpha m}C_{b,\alpha}(\alpha + b^{-2}))^s - 1,
\end{eqnarray*}
where $C_{b,\alpha} > 0$ is the constant in Lemma~\ref{lem_boundr}.
\end{theorem}

\begin{remark}\rm
By the lower bound of Sharygin~\cite{shar} we have that the worst-case error in the Korobov space $\HH_\alpha$ is at most $\landau(N^{-\alpha} (\log N)^{s-1})$. Hence it follows from Theorem~\ref{th_wce} that for a digital $(t,\alpha,\beta,m,s)$-net with $\beta > \alpha$ we must have $t = \landau((\beta-\alpha) m)$. Thus in order to avoid having a $t$-value which grows with $m$ we added the restriction $\beta \le \alpha$ in Definition~\ref{def_net}. Further, this also implies that a digital $(t,\alpha,\beta,s)$-sequence with $t < \infty$ cannot exist if $\beta > \alpha$, hence $\beta \le \alpha$ is in this case a consequence of the definition rather than a restriction.
\end{remark}
\begin{remark}\rm\label{rem_5}
Lemma~\ref{lem_wcesqrt} also holds for digital nets which are digitally shifted by an arbitrary digital shift $\bssigma\in [0,1)^s$ and hence it follows that Theorem~\ref{th_wce} also holds in a more general form, namely for all digital $(t,\alpha,\beta,m,s)$-net which are digitally shifted.
\end{remark}

Theorem~\ref{th_wce} shows that we can obtain the optimal
convergence rate for natural numbers $\alpha\ge 2$ by using a
digital $(t,\alpha,m,s)$-net. The constructions previously proposed
(for example by Sobol, Faure, Niederreiter or Niederreiter-Xing)
have only been shown to be $(t,1,m,s)$-nets and it has been proven
that they achieve a convergence of the worst-case error of
$\landau(N^{-1}(\log N)^{s -1})$.

We can use Theorem~\ref{th_wce} to obtain the following corollary.

\begin{corollary}\label{cor_bound}
Let $b$ be prime and let $C^{(d)}_1,\ldots, C^{(d)}_s\in \integer_b^{\infty \times \infty}$ be the generating matrices of  a digital $(t(a),a, \min(a,d),s)$-sequence $S$ over $\integer_b$ for any integer $a \ge 1$. Then for any real $\alpha \ge 1$ there is a constant $C'_{b,s,\alpha} > 0$, depending only on $b,s$ and $\alpha$, such that the worst-case error in the Korobov space $\HH_\alpha$ using the first $N = b^m$ points of $S$ is bounded by  $$e_{b,m,\alpha}(C_1^{(d)},\ldots, C_s^{(d)}) \le C'_{b,s,\alpha} b^{t(\lfloor \alpha\rfloor)} \frac{(\log N)^{s\lfloor \alpha\rfloor - 1}}{N^{\min(\lfloor \alpha \rfloor,d)}}.$$
\end{corollary}

\begin{remark}\rm
The above corollary shows that digital $(t,\alpha,
\min(\alpha,d),s)$-sequences constructed in
Section~\ref{sec_talphacons} achieve the optimal convergence (apart
from maybe some $\log N$ factor) of $P_{2\alpha}$ of
$\landau(N^{-2\alpha}(\log N)^{2s\alpha - 2})$ as long as $\alpha$
is an integer such that $1 \le \alpha \le d$. If  $\alpha > d$ we
obtain a convergence of $\landau(N^{-2d}(\log N)^{2s \alpha -2})$.
\end{remark}

\section{A bound on the mean square worst-case error in $\HH_\alpha$ for digital $(t,\alpha,\beta,m,s)$-nets and digital $(t,\alpha,\beta,s)$-sequences}\label{sec_boundrand}

To combine the advantages of random quadrature points with those of
deterministic quadrature points one sometimes uses a combination of
those two methods, see for example \cite{DP05, hick, matou, owen}.
The idea is to use a random element which preserves the essential
properties of a deterministic point set. We call the expectation
value of the square worst-case error of such randomized point sets
the mean square worst-case error.

\subsection{Randomization}

In the following we introduce a randomization scheme called digital
shift (see \cite{DP,matou}). Let $P_N = \{\bsx_0,\ldots,
\bsx_{N-1}\} \subseteq [0,1)^s$ with $\bsx_n = (x_{1,n},\ldots,
x_{s,n})$ and $x_{j,n} = x_{j,n,1} b^{-1} + x_{j,n,2} b^{-2} +
\cdots$ for $n = 0,\ldots, N-1$ and $j = 1,\ldots, s$. Let
$\sigma_{j,1}, \sigma_{j,2}, \ldots  \in \{0,1\}$ be i.i.d. for $j =
1,\ldots, s$. Then the randomly digitally shifted point set
$P_{N,\bssigma} = \{\bsz_0,\ldots, \bsz_{N-1}\}$, $\bsz_n =
(z_{1,n},\ldots, z_{s,n})$ using a digital shift, is then given by
$$z_{j,n} = (x_{j,n,1}\oplus \sigma_{j,1})b^{-1} +
(x_{j,n,2} \oplus \sigma_{j,2}) b^{-2} + \cdots$$ for $ j =
1,\ldots, s$ and $n = 0,\ldots, N-1$, where $x_{j,n,k} \oplus
\sigma_{j,n} = x_{j,n,k} + \sigma_{j,n} \pmod{b}$ (note that all
additions of the digits are carried out in the finite field
$\integer_b$). Subsequently let $P_N = \{\bsx_0,\ldots,
\bsx_{N-1}\}$ and let $P_{N,\bssigma}$ be the digitally shifted
point set $P_N$ using the randomization just described.

\subsection{The mean square worst-case error in the Korobov space}

In this section we will analyze the expectation value of
$e^2(P_{N,\bssigma},K_\alpha)$, which we  denote by
$\tilde{e}^2(P_N,K_\alpha) = \EE[e^2(P_{N,\bssigma},K_\alpha)]$,
with respect to the random digital shift described above. We call
$\tilde{e}^2(P_N,K_\alpha)$ the mean square worst-case error.

From (\ref{eq_wceker}) and  the linearity of the expectation
operator we have $$\tilde{e}^2(P_N,K_\alpha) =
\EE[e^2(P_{N,\bssigma},K_\alpha)] = -1 +
\frac{1}{N^2}\sum_{n,l=0}^{N-1} \prod_{j=1}^s
\EE[K_{\alpha}(z_{j,n},z_{j,l})].$$

In order to compute $\EE[K_{\alpha}(z_{j,n},z_{j,l})]$ we need the
following lemma, which, in a very similar form, was already shown
in \cite{DP05}, Lemma~3. Hence we omit a proof.
\begin{lemma}\label{lem_exber}
Let $x_1,x_2\in [0,1)$ and let $z_1,z_2 \in [0,1)$ be the points
obtained after applying an i.i.d. random digital shift to $x_1$ and
$x_2$. Then we have
$$\EE[\wal_k(z_1)\overline{\wal_l(z_2)}] = \left\{\begin{array}{ll}
\wal_k(x_1) \overline{\wal_k(x_2)} & \mbox{if } k = l,
\\ 0 & \mbox{otherwise}. \end{array}\right.$$
\end{lemma}

Recall that  $$K_\alpha(x_1,x_2) = \sum_{k,l = 0}^\infty
r_{b,\alpha}(k,l) \wal_k(x_1)\overline{\wal_l(x_2)},$$ where
$$r_{b,\alpha}(k,l) = \int_0^1 \int_0^1 K_\alpha(x_1,x_2)
\overline{\wal_k(x_1)} \wal_l(x_2) \rd x_1 \rd x_2.$$ Let $z_1,z_2$
be obtained by applying an i.i.d. random digital shift to $x_1,x_2$.
Using Lemma~\ref{lem_exber} and the linearity of expectation we
obtain $$\EE[K_{\alpha}(z_{1},z_{2})] = \sum_{k=0}^{\infty}
r_{b,\alpha}(k) \wal_k(x_1)\overline{\wal_k(x_2)},$$ where
$r_{b,\alpha}(k) = r_{b,\alpha}(k,k)$ and $r_{b,\alpha}(0) = 1$.

Therefore we obtain  $$\EE[e^2(P_{N,\bssigma},K_\alpha)] = -1 +
 \frac{1}{N^2}\sum_{n,l=0}^{N-1} \prod_{j=1}^s \sum_{k=0}^{\infty} r_{b,\alpha}(k)
\wal_k(x_{j,n}) \overline{\wal_k(x_{j,l})}.$$ Further we have $$
\prod_{j=1}^s \sum_{k=0}^{b^m-1} r_{b,\alpha}(k) \wal_k(x_{j,n})
\overline{\wal_k(x_{j,l})} = 1 + \sum_{\bsk\in\{0,\ldots, b^m-1\}^s
\setminus \{\bszero\}} r_{b,\alpha}(\bsk) \wal_k(\bsx_n \ominus
\bsx_l),$$ where we write $r_{b,\alpha}(\bsk) = \prod_{j=1}^s
r_{b,\alpha}(k_j)$ for $\bsk = (k_1,\ldots, k_s)$. We have shown the
following theorem.

\begin{theorem}\label{th_PN}
Let $b \ge 2$ be a natural number and let $\alpha > 1/2$ be a real
number. Then the mean square worst-case error for integration in the
Korobov space $\HH_\alpha$ using the point set $P_N$ randomized by a
digital shift is given by
\begin{eqnarray*}
\EE[e^2(P_{N,\bssigma},K_\alpha)] & = & \sum_{\bsk\in \NN_0^s
\setminus \{\bszero\}}
r_{b,\alpha}(\bsk)\frac{1}{N^2}\sum_{n,l=0}^{N-1} \wal_k(\bsx_n
\ominus \bsx_l).
\end{eqnarray*}
\end{theorem}
In the following we closer investigate the mean square worst-case
error for digital nets randomized with a digital shift.

Subsequently we will often write
$\tilde{e}^2_{b,m,\alpha}(C_1,\ldots, C_s)$ to denote the mean
square worst-case error $\EE[e(P_{b^m,\bssigma},K_\alpha)]$, where
$P_{b^m}$ is a digital net with generating matrices $C_1,\ldots,
C_s$ and $b^m$ points and $P_{b^m,\bssigma}$ is the digital net
$P_{b^m}$ randomized with a digital shift.

\begin{theorem}\label{th_q}
Let $m \ge 1$, $b$ be a prime number and $\alpha > 1/2$ be a real
number. The mean square worst-case error in the Korobov space
$\HH_\alpha$ using a randomly digitally shifted digital net over
$\integer_b$  with generating matrices $C_1,\ldots, C_s\in
\integer_b^{m\times m}$ is given by
$$\tilde{e}^2_{b,m,\alpha}(C_1,\ldots, C_s) =
\sum_{\bsk \in \D} r_{b,\alpha}(\bsk).$$
\end{theorem}

\begin{proof}
In \cite{DP} it was shown that $$\frac{1}{b^{2m}} \sum_{n,l =
0}^{b^m-1} \wal_{\bsk}(\bsx_n \ominus \bsx_l) = \frac{1}{b^{m}}
\sum_{n = 0}^{b^m-1} \wal_{\bsk}(\bsx_n) = \left\{\begin{array}{ll}
1 & \mbox{if } \bsk \in \D\cup \{\bszero\}, \\ 0 & \mbox{otherwise}.
\end{array}\right.$$ Hence the result follows from
Theorem~\ref{th_PN}.
\end{proof}

\begin{remark}\rm
Theorem~\ref{th_wcer} and Theorem~\ref{th_q} now imply that
$$\tilde{e}_{b,m,\alpha}(C_1,\ldots, C_s) = \sqrt{\sum_{\bsk \in \D}
r_{b,\alpha}(\bsk)} \le \sqrt{\sum_{\bsk,\bsl \in \D}
r_{b,\alpha}(\bsk,\bsl)} = e(P_{b^m},K_\alpha),$$ i.e. the root mean
square worst-case error is always smaller than the worst-case error,
see also Remark~\ref{rem_5}.
\end{remark}

Remark~\ref{rem_5} and also the above Remark imply that the bounds
on the worst-case error also hold for the root mean square
worst-case error. On the other hand, following the proofs for the
bound on the worst-case error using the criterium for the root mean
square worst-case error yields a better bound. We outline the
results subsequently.

Following the proof of Lemma~\ref{lem_q} we obtain
\begin{eqnarray*}
\sum_{\bsk \in\D} r_{b,\alpha}(\bsk) & \le & \sum_{\emptyset \neq u
\subseteq \{1,\ldots, s\}} (1 + b^{-2\alpha m}
\bar{C}_{b,\alpha}^2)^{s - |u|} (C_{b,\alpha}^2 +
\bar{C}_{b,\alpha}^2)^{|u|} \sum_{\bsk \in \D^\ast_{b^m,u}}
q^2_{b,\alpha}(\bsk) \\ &&  +  ( 1 + b^{-2\alpha m}
\bar{C}^2_{b,\alpha})^s - 1,
\end{eqnarray*}
where $C_{b,\alpha}$ is the constant from Lemma~\ref{lem_boundr} and
\begin{equation}\label{eq_cbar}
\bar{C}_{b,\alpha} = C_{b,\alpha} \sqrt{b^{-1} + (b^2-b)^{-1}
\prod_{c=3}^{\alpha + 1} (b^{2c}-b^{2(c-1)})^{-1}}.
\end{equation}

The sum $\sum_{\bsk \in \D^\ast_{b^m,u}} q^2_{b,\alpha}(\bsk)$ can
now be bounded using almost the same arguments as in the proof of
Lemma~\ref{lem_Q}. Doing this one can obtain that for a digital
$(t,\alpha,\beta,m,s)$-net we have
\begin{eqnarray*}
\sum_{\bsk \in \D^\ast_{b^m,u}} q^2_{b,\alpha}(\bsk) \le
(2b)^{|u|\alpha} b^{-2(\beta m - t) + 1} (\beta m - t + 1)^{|u|
\alpha - 1}.
\end{eqnarray*}
Hence we obtain the following theorem.

\begin{theorem}\label{th_wcerand}
Let $b$ be prime, $\alpha \ge 1$ an integer and let $C_1,\ldots, C_s
\in \integer_b^{m \times m}$ be the generating matrices of a digital
$(t,\alpha,\beta,m,s)$-net over $\integer_b$ with $m > t/\beta$.
Then the mean square worst-case error in the Korobov space
$\HH_\alpha$ is bounded by
\begin{eqnarray*}
\lefteqn{\tilde{e}^2_{b,m,\alpha}(C_1,\ldots, C_s) } \\ & \le &
 \frac{\left(1 +
b^{-2\alpha m} \bar{C}_{b,\alpha}^2 + (2b)^\alpha (C_{b,\alpha}^2 +
\bar{C}_{b,\alpha}^2) (\beta m - t + 1)^\alpha \right)^s - (1 +
b^{-2\alpha m} \bar{C}_{b,\alpha}^2)^s}{b^{2(\beta m - t) - 1}
(\beta m - t + 1)} \\ && + ( 1 + b^{-2\alpha m}
\bar{C}^2_{b,\alpha})^s - 1,
\end{eqnarray*}
where $C_{b,\alpha} > 0$ is the constant in Lemma~\ref{lem_boundr}
and the constant $\bar{C}_{b,\alpha} > 0$ is given by
(\ref{eq_cbar}).
\end{theorem}

We can use Theorem~\ref{th_wcerand} to obtain the following corollary.

\begin{corollary}\label{cor_bounderand}
Let $b$ be prime and let $C^{(d)}_1,\ldots, C^{(d)}_s\in
\integer_b^{\infty \times \infty}$ be the generating matrices of  a
digital $(t(a),a, \min(a,d),s)$-sequence $S$ over $\integer_b$ for
any integer $a \ge 1$. Then for any real $\alpha \ge 1$ there is a
constant $C''_{b,s,\alpha} > 0$, depending only on $b,s$ and
$\alpha$, such that the root mean square worst-case error in the
Korobov space $\HH_\alpha$ using the first $N = b^m$ points of $S$
is bounded by  $$\tilde{e}_{b,m,\alpha}(C_1^{(d)},\ldots, C_s^{(d)})
\le C''_{b,s,\alpha} b^{t(\lfloor \alpha\rfloor)} \frac{(\log
N)^{(s\lfloor \alpha\rfloor - 1)/2}}{N^{\min(\lfloor \alpha
\rfloor,d)}}.$$
\end{corollary}

\begin{remark}\rm
The above corollary shows that the digital $(t,\alpha,
\min(\alpha,d),s)$-sequences constructed in
Section~\ref{sec_talphacons} achieve the optimal convergence of
$P_{2\alpha}$ of $\landau(N^{-2\alpha}(\log N)^{s\alpha - 1})$ as
long as $\alpha$ is an integer such that $1 \le \alpha \le d$. (This
convergence is best possible for $\alpha = 1$ by the lower bound in
\cite{shar}.) If  $\alpha > d$ we obtain a convergence of
$\landau(N^{-2d}(\log N)^{s \alpha -1})$.
\end{remark}

Using the construction of Theorem~\ref{th_talphabeta} or
Theorem~\ref{th_talphabetaseq} it follows that $t(a)$ also depends
on the choice of $d$. Hence choosing a large value of $d$ also
increases the constant factor $b^{t(\lfloor \alpha \rfloor)}$ in
Corollary~\ref{cor_bound} and Corollary~\ref{cor_bounderand}.

\section{Some examples of digital $(t,\alpha,m,s)$-nets over
$\integer_2$}\label{sec_example}

In this section we give a simple example to show how the nets
described in this paper can be constructed. We use the construction
method outlined in Section~\ref{sec_talphacons}.

\subsection{Example of a digital $(0,2,m,1)$-net over $\integer_2$}

First we use the so-called Hammersley net as the underlying digital
net, which is a $(0,m,2)$-net over $\integer_2$. The generating
matrices for this net are given by
\begin{equation}\label{hamnet}
C_1 =  \begin{pmatrix}
                        1& 0 &\hdots&0\\
                        0&\ddots&\ddots&\vdots\\
                        \vdots&\ddots&\ddots&0\\
                        0&\hdots&0&1
                        \end{pmatrix}
                        \mbox{ and }
C_2 = \begin{pmatrix}
                        0&\hdots&0&1\\
                        \vdots&\rdots&\rdots&0\\
                        0&\rdots&\rdots&\vdots\\
                        1&0&\hdots&0
                        \end{pmatrix}.\end{equation}

Now we use the construction method of Section~\ref{sec_talphacons}
to construct the matrix $C_1^{(2)}$, i.e. $d = 2$ in this case. The
first row of $C_1^{(2)}$ is the first row of $C_1$, the second row
of $C_1^{(2)}$ is the first row of $C_2$, the third row of
$C_1^{(2)}$ is the second row of $C_1$, the fourth row of
$C_1^{(2)}$ is the second row of $C_2$ and so on. Assume that $C_1,
C_2$ are $m\times m$ matrices where $m$ is even. Then we obtain
$$C_1^{(2)} =  \begin{pmatrix}
                        1&0&\hdots& \hdots&\hdots&\hdots&\hdots&0\\
                        0&\hdots&\hdots& \hdots&\hdots&\hdots&0&1\\
                        0&1&0&\hdots&\hdots&\hdots&\hdots&0 \\
                        0&\hdots&\hdots&\hdots&\hdots&0&1&0\\
                        \hdots&\hdots&\hdots&\hdots&\hdots&\hdots&\hdots&\hdots \\
                        \hdots&\hdots&\hdots&\hdots&\hdots&\hdots&\hdots&\hdots \\
                        0&\hdots&0&1&0&\hdots&\hdots&0 \\
                        0&\hdots&\hdots&0&1&0&\hdots&0
                        \end{pmatrix}.
$$
So for example if $m = 4$ we obtain
\begin{equation}\label{c12}
C_1^{(2)} =  \begin{pmatrix}
                        1&0&0&0\\
                        0&0&0&1\\
                        0&1&0&0 \\
                        0&0&1&0
                        \end{pmatrix}.
\end{equation}

The matrix $C_1^{(2)}$ is of course non-singular and therefore the
point set one obtains are just equidistant points starting with $0$.

Assume that $m$ is even. Then the digital net which one obtains from
$C_1^{(2)}$ is a digital $(0,1,m,1)$-net over $\integer_2$ and, at
the same time, it is also a digital $(0,2,m,1)$-net. Note that using
the bound from Theorem~\ref{th_talphabeta} we obtain a $t$-value of
$1$, but by closer investigation using Definition~\ref{def_net} one
can see that the properties also hold for $t = 0$. Hence the
$t$-value obtained from Theorem~\ref{th_talphabeta} is not
necessarily strict even if the value of the underlying digital net
is strict.

\subsection{Example of a digital $(t,2,4,2)$-net over $\integer_2$}

Consider the digital $(1,4,4)$-net over $\integer_2$ with generating
matrices given by $C_1,C_2$ above and
$$C_3 =  \begin{pmatrix}
                        1& 1 &1&1\\
                        0&1&0&1\\
                        0&0&1&1\\
                        0&0&0&1
                        \end{pmatrix}
                        \mbox{ and }
C_4 = \begin{pmatrix}
                        0&1&1&0\\
                        1&1&0&1\\
                        0&0&0&1\\
                        0&0&1&0
                        \end{pmatrix}.$$
Then $C_1^{(2)}$ is given by (\ref{c12}) and $C_2^{(2)}$ is given by
\begin{equation*}
C_2^{(2)} =  \begin{pmatrix}
                        1&1&1&1\\
                        0&1&1&0\\
                        0&1&0&1 \\
                        1&1&0&1
                        \end{pmatrix}.
\end{equation*}

Using the digital construction scheme we obtain the points
\begin{align*}
(0,0),(\tfrac{1}{2},\tfrac{9}{16}),(\tfrac{1}{8},\tfrac{15}{16}),(\tfrac{5}{8},\tfrac{3}{8}),(\tfrac{1}{16},\tfrac{3}{4}),(\tfrac{9}{16},\tfrac{5}{16}),(\tfrac{3}{16},\tfrac{3}{16}),(\tfrac{11}{16},\tfrac{5}{8}),
\\
(\tfrac{1}{4},\tfrac{11}{16}),(\tfrac{3}{4},\tfrac{1}{8}),(\tfrac{3}{8},\tfrac{1}{4}),(\tfrac{7}{8},\tfrac{13}{16}),(\tfrac{5}{16},\tfrac{7}{16}),(\tfrac{13}{16},\tfrac{7}{8}),(\tfrac{7}{16},\tfrac{1}{2}),(\tfrac{15}{16},\tfrac{1}{16}).
\end{align*}


It can be checked that this digital net is a digital
$(1,1,4,2)$-net, i.e. a digital $(1,4,2)$-net (the first two rows of
$C_1^{(2)}$ and the first two rows of $C_2^{(2)}$ are linearly
dependent, so the $t$-value cannot be $0$ when $\alpha = 1$).

Now we investigate the $t$-value when $\alpha = 2$. First note that
Theorem~\ref{th_talphabeta} yields a $t$-value of $4$ for $\alpha =
2$ ($d = s = 2$). Further the $t$-value cannot be $2$ in this case:
we need to consider all cases where $i_{1,1} + i_{1,\min(\nu_1,2)} +
i_{2,1} + i_{2,\min(\nu_2,2)} \le \alpha m - t = 2 \cdot 4 - 2 = 6$
with $0 \le \nu_1,\nu_2 \le 4$. But by choosing $i_{1,1} = i_{2,1} =
2$ and $i_{1,2} = i_{2,2} = 1$ we obtain the first two rows of $C_1$
and the first two rows of $C_2$, and as those $4$ rows are linearly
dependent it follows that the $t$-value cannot be $2$. Now let us
check whether a $t$-value of $3$ is possible: we need to have
$i_{1,1} + i_{1,\min(\nu_1,2)} + i_{2,1} + i_{2,\min(\nu_2,2)} \le
5$, hence $\nu_1,\nu_2 \ge 2$ is not possible (because then we would
have $i_{1,1} + i_{1,2} + i_{2,1} + i_{2,2} \ge 2 + 1 + 2 + 1
> 5$). Further the conditions are satisfied if either $\nu_1 = 0$ or
$\nu_2 = 0$ as the matrices $C_1^{(2)}$ and $C_2^{(2)}$ are
non-singular. If $\nu_1 > 2$ then $i_{1,1} \ge 3$ and $i_{1,2} \ge
2$ and hence $i_{1,1} + i_{1,2} \ge 5$ and we can only get $i_{1,1}
+ i_{1,\min(\nu_1,2)} + i_{2,1} + i_{2,\min(\nu_2,2)} \le 5$ if
$\nu_2 = 0$. Hence if either $\nu_1 > 2$ or $\nu_2 > 2$ the
properties are also satisfied. Thus we are left with the following
three cases: $(\nu_1,\nu_2) = (1,1)$, $(\nu_1,\nu_2) = (1,2)$ and
$(\nu_1,\nu_2) = (2,1)$.

Now let $\nu_1 = \nu_2 = 1$. Then we need to take one row of each
matrix $C_1^{(2)}$ and $C_2^{(2)}$ such that the sum of their row
indices is smaller or equal to $5$ and check whether those two rows
are linearly independent. It can be checked that this is always the
case: let $C_j^{(2)} = (c_{j,1}^\top, c_{j,2}^\top,
c_{j,3}^\top,c_{j,4}^\top)$, i.e. $c_{j,k}$ denotes the $k$-th row
of $C_{j}^{(2)}$. Then the pairs of vectors $(c_{1,k},c_{2,l})$
where $k + l \le 5$ are always linearly independent for all
admissible choices of $k$ and $l$ (i.e. $c_{1,k} \neq c_{2,l}$).

Consider now $\nu_1 = 1$ and $\nu_2 = 2$, i.e. we take one row from
$C_1^{(2)}$ and two rows from $C_2^{(2)}$ such that the sum of the
row indices does not exceed $5$. Note that $i_{2,2}$ has to be $1$
otherwise $i_{2,1} + i_{2,2} \ge 5$ and $i_{1,1}$ cannot even be
$1$. As $i_{1,1} \ge 1$ and $i_{2,1} \ge 2$ the only choices left
are $i_{1,1} = 1$ and $i_{2,1} = 2,3$ and $i_{1,1} = 2$ and $i_{2,1}
= 2$. So we need to check whether the triplets $(c_{1,1}, c_{2,1},
c_{2,2})$, $(c_{1,1},c_{2,1},c_{2,3})$ and $(c_{1,2}, c_{2,1},
c_{2,2})$ are all linearly independent, which upon inspection can be
seen to be the case.

The case $\nu=2$ and $\nu = 1$ can also be checked as the previous
case. In this case all the relevant sets of vectors are also always
linearly independent, hence a $t$-value of $3$ is possible for
$\alpha = 2$, i.e. the digital net above is a (strict) digital
$(3,2,4,2)$-net.

The classical $t$-value (i.e. $\alpha = 1$) of this digital net is
not as good as for example the $t$-value of the Hammersley net
(which is $0$). On the other hand it can be checked that for $\alpha
= 2$ the $t$-value of the Hammersley net where $m = 4$ is $4$ and
hence for this case it is worse than the $t$-value of the digital
net constructed above.

As a last example let us consider the Hammersley net again for
arbitrary $m \ge 1$, i.e. with the $m \times m$ generating matrices
given by (\ref{hamnet}). As for example the first row of $C_1$ and
the last row of $C_2$ are the same (and therefore linearly
dependent) we must have $\beta m - t < m+1$ for all $\alpha \ge 1$
(for $\alpha =1$ we can still choose $\beta = 1$ and $t = 0$ and
hence the Hammersley net achieves the optimal $t$-value, but for
$\alpha > 1$ we have seen in Section~\ref{sec_talphacons} that there
are better constructions). It is sensible to choose $\beta$ such
that we can have a $t$-value which is independent of $m$ (for
example this is the case when one considers sequences and which is
also the motivation for introducing those parameters; for digital
nets it would of course also make sense to just state the value of
$\beta m - t$ and $m$ instead of $t,\beta$ and $m$). This means that
$\beta \le 1$, and as $\beta$ indicates the convergence rate one can
obtain it follows that one cannot expect to obtain a convergence
rate beyond $(b^m)^{-1+\delta}$ (for an arbitrary small $\delta >
0$) when using a Hammersley net.

\section{Appendix: Some lemmas}

We need the following lemmas.

\begin{lemma}\label{lem_int}
Let $j \ge 1$, $a \ge 0$, $b \ge 2$ and $0 \le u,v < b^a$ with $u
\neq v$. Then we have
\begin{equation*}
\int_{u/b^a}^{(u+1)/b^a} \int_{u/b^a}^{(u+1)/b^a} |x-y|^j \rd x \rd
y = \frac{2}{b^{a(j+2)} (j+1)(j+2)}
\end{equation*}
and
\begin{equation*}
\int_{u/b^a}^{(u+1)/b^a} \int_{v/b^a}^{(v+1)/b^a} |x-y|^j \rd x \rd
y =  \frac{2 j!}{b^{a(j+2)}} \sum_{l=0}^{\lfloor j/2 \rfloor}
\frac{|u-v|^{j-2l}}{(j-2l)! (2l+2)!}.
\end{equation*}
\end{lemma}

\begin{proof}
We have
\begin{eqnarray*}
\int_{u/b^a}^{(u+1)/b^a} \int_{u/b^a}^{(u+1)/b^a} |x-y|^j \rd x \rd
y & = & \int_{0}^{1/b^a} \int_{0}^{1/b^a} |x-y|^j \rd x \rd y \\ & =
& \frac{1}{b^{a(j+2)}} \int_0^1 \int_0^1 |x-y|^j \rd x \rd y.
\end{eqnarray*}
We divide the last double integral in two parts, we have $$ \int_0^1
\int_0^1 |x-y|^j \rd x \rd y = \int_0^1 \int_0^y (y-x)^j \rd x \rd y
+ \int_0^1 \int_y^1 (x-y)^j \rd x \rd y.$$ We calculate the first
part and obtain $$\int_0^1 \int_0^y (y-x)^j \rd x \rd y =
\frac{1}{j+1} \int_0^1 y^{j+1} \rd y = \frac{1}{(j+1)(j+2)}$$ and
the second part is given by
$$\int_0^1 \int_y^1 (x-y)^j \rd x \rd y = \frac{1}{j+1} \int_0^1 (1-y)^{j+1} \rd y = \frac{1}{(j+1)(j+2)}.$$ Hence
we have $$\int_0^1 \int_0^1 |x-y|^j \rd x \rd y =  \frac{2}{(j+1)
(j+2)}.$$

For the second part we have
\begin{eqnarray*}
\int_{u/b^a}^{(u+1)/b^a} \int_{v/b^a}^{(v+1)/b^a} |x-y|^j \rd x \rd
y &=& \int_0^{1/b^a} \int_{|u-v|/b^a}^{(|u-v|+1)/b^a} |x-y|^j \rd x
\rd y \\  &=& \frac{1}{b^{a(j+2)}} \int_0^1 \int_{|u-v|}^{|u-v|+1}
(x-y)^j \rd x \rd y,
\end{eqnarray*}
where now $|u-v| \ge 1$. We have
\begin{eqnarray*}
\int_0^1 \int_{|u-v|}^{|u-v|+1} (x-y)^j \rd x \rd y &=&
\frac{1}{j+1} \int_0^1 \left( (|u-v|+1-y)^{j+1} - (|u-v|-y)^{j+1}
\right) \rd y \\ & = & \frac{2|u-v|^{j+2} - (|u-v|+1)^{j+2} -
(|u-v|-1)^{j+2}}{(j+1)(j+2)}.
\end{eqnarray*}
The result follows by simplifying the sum in the numerator.
\end{proof}

\begin{lemma}\label{lem_Ijeven}
Let $k \ge 1$ be given by $k = \kappa_{a_1-1} b^{a_1-1} + \cdots +
\kappa_{a_\nu-1} b^{a_\nu-1}$ for some $\nu \ge 1$,
$\kappa_{a_1-1},\ldots, \kappa_{a_\nu-1} \in \{1,\ldots, b-1\}$ and
$1 \le a_\nu < \cdots < a_1$. For any  even $0 \le j < 2\nu$ we have
$I_j(k) = 0$.
\end{lemma}

\begin{proof}
The result for $j = 0$ follows from Proposition~\ref{prop1} and
(\ref{eq_Ijk}). It was shown in \cite{DP}, Appendix A, that $$x =
\frac{1}{2} + \sum_{c=1}^\infty \sum_{\tau =1}^{b-1} \frac{1}{b^c
(\de^{-2\pi \icomp \tau/b}-1)} \wal_{\tau b^{c-1}}(x)$$ and hence
\begin{eqnarray*}
|x-y|^j &=& \left(\sum_{c=1}^\infty \sum_{\tau =1}^{b-1}
\frac{1}{b^{c}(\de^{-2\pi\icomp \tau /b}-1)} (\wal_{\tau b^{c-1}}(y)
- \wal_{\tau b^{c-1}}(x)) \right)^j \\ &=& \sum_{c_1,\ldots, c_j =
1}^\infty \frac{1}{b^{c_1+\cdots + c_j}} \prod_{i=1}^j \sum_{\tau
=1}^{b-1} \frac{\wal_{\tau b^{c_i-1}}(y)-\wal_{\tau
b^{c_i-1}}(x)}{\de^{-2\pi\icomp \tau /b}-1}.
\end{eqnarray*}
Let $$A_k(c_1,\ldots, c_j) = \int_0^1 \int_0^1  \prod_{i=1}^j
\sum_{\tau =1}^{b-1} \frac{\wal_{\tau b^{c_i-1}}(y)-\wal_{\tau
b^{c_i-1}}(x)}{\de^{-2\pi\icomp \tau /b}-1} \overline{\wal_k(x)}
\wal_k(y) \rd x \rd y.$$ Then we have  $$I_j(k) = \sum_{c_1,\ldots,
c_j = 1}^\infty \frac{A_k(c_1,\ldots, c_j)}{b^{c_1+\cdots + c_j}}.$$
We have
\begin{eqnarray*}
\lefteqn{ \prod_{i=1}^j \sum_{\tau =1}^{b-1} \frac{\wal_{\tau
b^{c_i-1}}(y)-\wal_{\tau b^{c_i-1}}(x)}{\de^{-2\pi\icomp \tau /b}-1}
} \\ &=& \sum_{\tau_1,\ldots, \tau_j = 1}^{b-1} \prod_{i=1}^j
(\de^{-2\pi \icomp \tau_j/b}-1)^{-1} \sum_{u \subseteq \{1,\ldots,
j\}} (-1)^{|u|} \prod_{i \in u} \wal_{\tau_i b^{c_i-1}}(y) \prod_{i
\not\in u} \wal_{\tau_i b^{c_i-1}}(x) \\ &=& \sum_{\tau_1,\ldots,
\tau_j = 1}^{b-1} \prod_{i=1}^j (\de^{-2\pi \icomp \tau_j/b}-1)^{-1}
\sum_{u \subseteq \{1,\ldots, j\}} (-1)^{|u|} \wal_{C_{u,\tau}}(y)
\wal_{C_{\{1,\ldots, j\} \setminus u}, \tau}(x),
\end{eqnarray*}
where $C_{u, \tau} = \sum_{i \in u} \tau_i b^{c_i-1}$ and hence
\begin{eqnarray*}
\lefteqn{ A_k(c_1,\ldots, c_j) } \\ &=& \sum_{\tau_1,\ldots, \tau_j
= 1}^{b-1} \prod_{i=1}^j (\de^{-2\pi \icomp \tau_j/b}-1)^{-1}
\sum_{u \subseteq \{1,\ldots, j\}} (-1)^{|u|} \\ && \int_0^1
\int_0^1\wal_{C_{u,\tau}}(y) \wal_{C_{\{1,\ldots, j\} \setminus u,
\tau}}(x) \overline{\wal_k(x)} \wal_k(y) \rd x \rd y \\ &=&
\sum_{\tau_1,\ldots, \tau_j = 1}^{b-1} \prod_{i=1}^j (\de^{-2\pi
\icomp \tau_j/b}-1)^{-1} \\ && \sum_{u \subseteq \{1,\ldots, j\}}
(-1)^{|u|}  \int_0^1\wal_{C_{u,\tau} \oplus k}(y) \rd y \int_0^1
\wal_{C_{\{1,\ldots, j\} \setminus u,\tau} \ominus k}(x) \rd x.
\end{eqnarray*}

Note that if $\nu > j/2$ we either have $C_{u,\tau} \oplus k \neq 0$
or $C_{\{1,\ldots, j\} \setminus u, \tau} \ominus k \neq 0$ and
hence $A_k(c_1,\ldots, c_j) = 0$. The result now follows.
\end{proof}

Let $\sigma_p(n)  = \sum_{h=1}^{n-1} h^p$. It is known that
\begin{equation}\label{eq_sum}
\sigma_p(n) = \sum_{h=0}^p \frac{B_h}{h!} \frac{p!}{(p+1-h)!}
n^{p+1-h},
\end{equation}
where $B_0, B_1, \ldots$ are the Bernoulli numbers (in particular,
$B_0 = 1$, $B_1 = -1/2$ and $B_2 = 1/6$).

\begin{lemma}\label{lem_doublesum}
Let $b \ge 2$, $1 \le d \le a $, $k = \kappa_{d-1} b^{d-1} + \cdots
+ \kappa_0$ where $\kappa_{d-1} \in \{1,\ldots, b-1\}$,
$\kappa_{d-2},\ldots, \kappa_0 \in \{0,\ldots, b-1\}$, $m = m_{a-1}
b^{a-1} + \cdots + m_0$ and $n = n_{a-1} b^{a-1} + \cdots + n_0$.
Then we have
\begin{equation*}
 \sum_{n=0}^{b^{a}-2} \sum_{m=n+1}^{b^{a}-1} \wal_k((n \ominus m)/b^a) =  b^{2a-d}\left(\frac{1}{2} +  \frac{1}{\de^{2\pi \icomp \kappa_{d-1}/b}-1}\right) - \frac{b^a}{2},
\end{equation*}
\begin{equation*}
 \sum_{n=0}^{b^{a}-2} \sum_{m=n+1}^{b^{a}-1} (m-n) \wal_k((n \ominus m)/b^a) =  b^{3a-2d} \left(\frac{1}{6} -\frac{1}{2\sin^2(\kappa_{d-1}\pi/b)}\right) - \frac{b^a}{6}
\end{equation*}
and
\begin{equation*}
 \sum_{n=0}^{b^{a}-2} \sum_{m=n+1}^{b^{a}-1} (m-n) = \frac{1}{6} (b^{3a} - b^a).
\end{equation*}
\end{lemma}

\begin{proof}
In order to obtain a formula for the first sum, let $m' = m_{a-1}
b^{a-1} + \cdots + m_{a-d+1} b^{a-d+1}$, $m'' = m_{a-d-1} b^{a-d-1}
+ \cdots + m_0$, $n' = n_{a-1} b^{a-1} + \cdots + n_{a-d+1}
b^{a-d+1}$ and $n'' = n_{a-d-1} b^{a-d-1} + \cdots + n_0$. First
consider the case where $m' > n'$ and arbitrary $m'',n''$. We have
\begin{equation*}
\sum_{n_{a-d}=0}^{b-1} \sum_{m_{a-d}=0}^{b-1} \de^{2\pi
\icomp(\kappa_0(n_{a-1} - m_{a-1}) + \cdots + \kappa_{d-1} (n_{a-d}
- m_{a-d}))/b}  = 0,
\end{equation*}
as $\sum_{m=0}^{b-1} \de^{2\pi \icomp \kappa m/b} = 0$ for all
$\kappa = 1,\ldots, b-1$. Thus we only need to consider the case
where $m'  = n'$, for which case we have $$\de^{2 \pi \icomp
(\kappa_0(n_{a-1} - m_{a-1}) + \cdots + \kappa_{d-1} (n_{a-d} -
m_{a-d}))/b} = \de^{2 \pi \icomp \kappa_{d-1} (m_{a-d} -
n_{a-d})/b}.$$ This part is now given by
\begin{equation}\label{sum2}
b^{d-1}\sum_{n''=0}^{b^{a-d}-1} \sum_{m''=0}^{b^{a-d}-1}
\sum_{n_{a-d} = 0}^{b-1} \sum_{m_{a-d}=0}^{b-1}  \de^{2 \pi \icomp
\kappa_{d-1}(n_{a-d} -m_{a-d})/b},
\end{equation}
where we have the additional assumption $m_{a-d} b^{a-d} + m'' >
n_{a-d} b^{a-d} + n''$. First consider the case where $m_{a-d} >
n_{a-d}$. This part of (\ref{sum2}) is given by
\begin{eqnarray*}
\lefteqn{ b^{d-1} \sum_{n''=0}^{b^{a-d}-1} \sum_{m''=0}^{b^{a-d}-1}
\sum_{n_{a-d}=0}^{b-2} \sum_{m_{a-d}=n_{a-d}+1}^{b-1}  \de^{2 \pi
\icomp \kappa_{d-1}(n_{a-d} -m_{a-d})/b}  } \qquad\qquad\qquad\qquad
\\ & = & b^{d-1} b^{2(a-d)}  \sum_{n_{a-d}=0}^{b-2}
\sum_{m_{a-d}=n_{a-d}+1}^{b-1} \de^{2 \pi \icomp
\kappa_{d-1}(n_{a-d} -m_{a-d})/b} \\ & = & \frac{b^{2a-d}}{\de^{2\pi
\icomp \kappa_{d-1}/b}-1}.
\end{eqnarray*}
Now consider the case where $m_{a-d} = n_{a-d}$. In this case we
have the assumption that $m'' > n''$ and hence this part of
(\ref{sum2}) is given by $$b^d \sum_{n''=0}^{b^{a-d}-2} \sum_{m'' =
n''+1}^{b^{a-d}-1} 1 = \frac{1}{2} \left(b^{2a-d} - b^a \right).$$
Thus (\ref{sum2}) is given by $$\frac{b^{2a-d}}{\de^{2\pi \icomp
\kappa_{d-1}/b}-1} + \frac{1}{2} \left(b^{2a-d} - b^a \right)$$ and
the first result follows.

For the second sum let again $m' = m_{a-1} b^{a-1} + \cdots +
m_{a-d+1} b^{a-d+1}$, $m'' = m_{a-d-1} b^{a-d-1} + \cdots + m_0$,
$n' = n_{a-1} b^{a-1} + \cdots + n_{a-d+1} b^{a-d+1}$ and also $n''
= n_{a-d-1} b^{a-d-1} + \cdots + n_0$. First consider the case where
$m' > n'$ and arbitrary $m'',n''$. We have
\begin{eqnarray*}
\lefteqn{\sum_{n_{a-d}=0}^{b-1} \sum_{m_{a-d}=0}^{b-1}  (m-n)
\de^{2\pi \icomp(\kappa_0(n_{a-1} - m_{a-1}) + \cdots + \kappa_{d-1}
(n_{a-d} - m_{a-d}))/b} } \\ & = &  \sum_{n_{a-d}=0}^{b-1}
\sum_{m_{a-d}=0}^{b-1}  (m_{a-d}-n_{a-d}) \de^{2 \pi \icomp
(\kappa_0(n_{a-1} - m_{a-1}) + \cdots + \kappa_{d-1} (n_{a-d} -
m_{a-d}))/b} \\ &= & 0,
\end{eqnarray*}
as $\sum_{m=0}^{b-1} \de^{2\pi \icomp \kappa m/b} = 0$ for all
$\kappa = 1,\ldots, b-1$.

Thus we are left with the case where $m' = n'$. We have $$\de^{2 \pi
\icomp (\kappa_0(n_{a-1} - m_{a-1}) + \cdots + \kappa_{d-1} (n_{a-d}
- m_{a-d}))/b} = \de^{2 \pi \icomp \kappa_{d-1} (m_{a-d} -
n_{a-d})/b}.$$ Hence this part is given by
\begin{equation}\label{sum1}
b^{d-1} \sum_{n_{a-d} = 0}^{b-1} \sum_{m_{a-d}=0}^{b-1}
\sum_{n''=0}^{b^{a-d}-1} \sum_{m''=0}^{b^{a-d}-1} (m''-n'' +
b^{a-d}(m_{a-d}-n_{a-d})) \de^{2 \pi \icomp \kappa_{d-1}(n_{a-d}
-m_{a-d})/b},
\end{equation}
where we have the additional assumption $m_{a-d} b^{a-d} + m'' >
n_{a-d} b^{a-d} + n''$. First consider the case where $m_{a-d} >
n_{a-d}$. This part of (\ref{sum1}) is given by
\begin{eqnarray*}
\lefteqn{b^{d-1} \sum_{0 \le n_{a-d} < m_{a-d} < b}\sum_{m'',
n''=0}^{b^{a-d}-1} (m''-n'' + b^{a-d}(m_{a-d}-n_{a-d}))  \de^{2 \pi
\icomp \kappa_{d-1}(n_{a-d} -m_{a-d})/b} } \quad\qquad\qquad
\\ & = & b^{d-1} b^{3(a-d)}  \sum_{n_{a-d}=0}^{b-2} \sum_{m_{a-d}=n_{a-d}+1}^{b-1} (m_{a-d}-n_{a-d}) \de^{2 \pi \icomp \kappa_{d-1}(n_{a-d} -m_{a-d})/b} \\ & = & -\frac{b^{3a-2d}}{2\sin^2(\kappa_{d-1}\pi/b)}.
\end{eqnarray*}
Now consider the case where $m_{a-d} = n_{a-d}$. In this case we
have the assumption that $m'' > n''$ and hence this part of
(\ref{sum1}) is given by
$$b^d \sum_{n''=0}^{b^{a-d}-2} \sum_{m'' = n''+1}^{b^{a-d}-1}
(m''-n'') =\frac{b^d}{6} (b^{3(a-d)} - b^{a-d}).$$ (This result can
be obtained using (\ref{eq_sum}), see the proof of the third part
below.) Thus (\ref{sum1}) is given by
$$-\frac{b^{3a-2d}}{2\sin^2(\kappa_{d-1}\pi/b)} + \frac{b^d}{6}
(b^{3(a-d)} - b^{a-d})$$ and the second result follows.

The third result can easily be verified by using (\ref{eq_sum}).
Indeed we have
\begin{equation*}
 \sum_{n=0}^{b^{a}-2} \sum_{m=n+1}^{b^{a}-1} (m-n)  =  \sum_{n=0}^{b^a-2} \sum_{m=1}^{b^a-1-n} m =
 \sum_{n=0}^{b^a-2} \sigma_1(b^a-n) = \frac{1}{2} \sum_{n=0}^{b^a-2} ((b^a-n)^2-(b^a-n)).
\end{equation*}
The last sum can be written as $\tfrac{1}{2} \sum_{n=1}^{b^a}
(n^2-n) = \tfrac{1}{2} (\sigma_2(b^a+1) - \sigma_1(b^a+1))$ and by
using (\ref{eq_sum}) again the result follows.
\end{proof}

\begin{lemma}\label{lem_doublesum2}
Let $j \ge 0$, $\nu \ge 1$, $1 \le a_\nu < \cdots < a_1 \le a $, $k
=  \kappa_{a_1-1} b^{a_1-1} + \cdots + \kappa_{a_\nu-1} b^{a_\nu-1}$
where $\kappa_{a_1-1},\ldots, \kappa_{a_\nu-1} \in \{1,\ldots,
b-1\}$. Then we have
\begin{equation*}
 \sum_{n=0}^{b^{a}-2} \sum_{m=n+1}^{b^{a}-1} (m-n)^{j} = b^a \sigma_j(b^a) - \sigma_{j+1}(b^a) \le \frac{b^{a(j+2)}}{(j+1)(j+2)}
\end{equation*}
and
\begin{eqnarray*}
\left|\sum_{n=0}^{b^{a}-2} \sum_{m=n+1}^{b^{a}-1} (m-n)^{j}
\wal_k((n\ominus m)/b^a)\right| & \le & C_{b,j} b^{(j+2) a - 2 (a_1
+ \cdots + a_{\min(\nu,\lceil j/2\rceil)})}
\end{eqnarray*}
for some constant $C_{b,j} > 0$ which is independent of $\nu$, $a$
and $a_1,\ldots, a_\nu$.
\end{lemma}

\begin{proof}
We have $$\sum_{n=0}^{b^{a}-2} \sum_{m=n+1}^{b^{a}-1} (m-n)^{j}  =
\sum_{n = 1}^{b^a-1} (b^a-n) n^j = b^a \sigma_{j}(b^a) -
\sigma_{j+1}(b^a),$$ and by using (\ref{eq_sum}) it follows that
\begin{eqnarray*}
\lefteqn{ b^a \sigma_{j}(b^a) - \sigma_{j+1}(b^a)} \\ &=& b^{a(j+2)}
\left(\sum_{h=0}^j B_h \left(\frac{j!}{h!(j+1-h)!} -
\frac{(j+1)!}{h!(j+2-h)!}\right)b^{-ah} - B_{j+1} b^{-a(j+1)}\right)
\\ & \le & b^{a(j+2)} B_0 \frac{1}{(j+1)(j+2)},
\end{eqnarray*}
from which the first part follows as $B_0 = 1$.

For $j = 0,1$ the second part immediately follows from
Lemma~\ref{lem_doublesum}. Let now $j \ge 2$ and assume the result
holds for all $j-1,\ldots, 1,0$.

Let $m = m_{a-1} b^{a-1} + \cdots + m_0$ and $n = n_{a-1} b^{a-1} +
\cdots + n_0$. In order to obtain a bound on
\begin{equation}\label{doublesum}
\left| \sum_{n=0}^{b^{a}-2} \sum_{m=n+1}^{b^{a}-1} (m-n)^{j}
\de^{2\pi \icomp ( \kappa_{a_1-1}(n_{a-a_1} - m_{a-a_1}) + \cdots +
\kappa_{a_\nu-1}(n_{a-a_\nu} - m_{a-a_\nu}))/b} \right|
\end{equation}
we first sum over the digits $m_{a-a_1}$ and $n_{a-a_1}$.

Let $m' = m_{a-1} b^{a-1} + \cdots + m_{a-a_1+1} b^{a-a_1+1}$, $n' =
n_{a-1} b^{a-1} + \cdots + n_{a-a_1+1} b^{a-a_1+1}$, $m'' =
m_{a-a_1-1} b^{a-a_1-1} + \cdots + m_0$ and $n'' = n_{a-a_1-1}
b^{a-a_1-1} + \cdots + n_0$. We consider two cases, namely where $m'
> n'$ and where $m' = n'$.

For $m' = n'$ we either have $m_{a-a_1} > n_{a-a_1}$ or $m_{a-a_1} =
n_{a-a_1}$ and $m'' > n''$, as $m > n$. First let $m_{a-a_1} >
n_{a-a_1}$. We have $b^{a_1-1}$ choices for $m'=n'$ and the sum over
the digits $m_{a-a_1}, n_{a-a_1}$ with $m_{a-a_1} > n_{a-a_1}$ can
be written as one sum so that the part of (\ref{doublesum}) where
$m' = n'$ is given by
\begin{eqnarray*}
\lefteqn{ b^{a_1-1} \left|\sum_{n''=0}^{b^{a-a_1}-1}
\sum_{m''=0}^{b^{a-a_1}-1} \sum_{\tau=1}^{b-1} (b-\tau) (\tau
b^{a-a_1} + m''-n'')^{j} \de^{-2\pi \icomp \kappa_{a_1-1}\tau /b}
\right| } \\ & \le &  b^{a_1-1} \sum_{n''=0}^{b^{a-a_1}-1}
\sum_{m''=0}^{b^{a-a_1}-1} \sum_{\tau=1}^{b-1} (b-\tau) (\tau
b^{a-a_1} + m''-n'')^{j} \\ & \le & C''_{b,j} b^{a_1}
b^{(j+2)(a-a_1)},
\end{eqnarray*}
for some constant $C''_{b,j} > 0$ which only depends on $b$ and $j$.
Hence this part satisfies the bound. Now let $m_{a-a_1}= n_{a-a_1}$,
then we have $m'' > n''$ and hence  the part of (\ref{doublesum})
where $m' = n'$ and $m_{a-a_1}=n_{a-a_1}$ is given by
\begin{equation*}
b^{a_1} \sum_{n''=0}^{b^{a-a_1}-1} \sum_{m''=n''+1}^{b^{a-a_1}-1}
(m''-n'')^{j} \le  \frac{b^{a_1}b^{(j+2)(a-a_1)}}{(j+1)(j+2)},
\end{equation*}
where the inequality was already obtained in the first part of this
proof. Hence also this part satisfies the bound.

Now we consider the part of (\ref{doublesum}) where $m' > n'$. We
have
\begin{eqnarray}\label{summe2}
\lefteqn{\sum_{m_{a-a_1}, n_{a-a_1} = 0}^{b-1} \!\!\!\!\! (m'-n' +
b^{a-a_1} (m_{a-a_1}-n_{a-a_1}) + m'' - n'')^j \de^{2\pi \icomp
\kappa_{a_1-1}({n_{a-a_1} - m_{a-a_1}})/b } } \nonumber \\ &=& b
(m'-n'+m''-n'')^j + \sum_{\tau = 1}^{b-1} (b-\tau) [\de^{-2\pi
\icomp \kappa_{a_1-1} \tau /b } (m'-n' + \tau b^{a-a_1} + m'' -
n'')^j \nonumber \\ && + \de^{2\pi \icomp \kappa_{a_1-1} \tau /b }
(m'-n' - \tau b^{a-a_1} + m'' - n'')^j] \nonumber \\ & = &  b
(m'-n'+m''-n'')^j  + \sum_{u =0}^j {j \choose u} (m'-n' +
m''-n'')^{j-u} b^{u(a-a_1)} E_u,
\end{eqnarray}
where $$E_u = \sum_{\tau = 1}^{b-1} (b-\tau) [\de^{-2\pi \icomp
\kappa_{a_1-1} \tau /b } \tau^{u} +  \de^{2\pi \icomp \kappa_{a_1-1}
\tau /b } (-\tau)^{u}].$$ It can be checked that $E_0 = -b$ and $E_1
= 0$. Hence (\ref{summe2}) is given by $$ \sum_{u = 2}^j {j \choose
u} (m'-n' + m''-n'')^{j-u} b^{u(a-a_1)} E_u,$$ and hence the result
follows from the induction assumption or the first part.
\end{proof}

\begin{lemma}\label{lem_Iodd}
Let $k \ge 1$ be given by $k = \kappa_{a_1-1} b^{a_1-1} + \cdots +
\kappa_{a_\nu-1} b^{a_\nu-1}$ for some $\nu \ge 1$, $1 \le a_\nu <
\cdots < a_1$ and $\kappa_{a_1-1}, \ldots, \kappa_{a_\nu-1} \in
\{1,\ldots, b-1\}$. Then for $j \ge 1$ we have $$|I_j(k)| \le
\frac{\bar{C}_{b,j}}{b^{2(a_1 + \cdots + a_{\min(\nu,\lceil
j/2\rceil )})}}$$ for some constant $\bar{C}_{b,j} > 0$ which
depends only on $b$ and $j$.
\end{lemma}

\begin{proof}
Let $k = \kappa_{a-1} b^{a-1} + \cdots + \kappa_0$, where now $a =
a_1$, $u = u_{a-1} b^{a-1} + \cdots + u_0$ and $v = v_{a-1} b^{a-1}
+ \cdots + v_0$. Then we have
\begin{eqnarray*}
I_j(k) & = & \int_0^1 \int_0^1 |x-y|^j \overline{\wal_{k}(x) }
\wal_{k}(y) \rd x \rd y \\ & = & \sum_{u = 0}^{b^a-1}
\sum_{v=0}^{b^a-1} \de^{2\pi \icomp (\kappa_0 (u_{a-1} - v_{a-1}) +
\cdots + \kappa_{a-1} (u_0 - v_0))} \int_{u/b^a}^{(u+1)/b^a} \!\!
\int_{v/b^a}^{(v+1)/b^a} \!\! |x-y|^j \rd x \rd y.
\end{eqnarray*}
For $u = v$ we have  $\de^{2\pi \icomp (\kappa_0 (u_{a-1} - v_{a-1})
+ \cdots + \kappa_{a-1} (u_0 - v_0))} = 1$. Using
Lemma~\ref{lem_int} it follows that this part in the above sum is
given by $$\frac{2}{b^{a(j+1)}(j+1)(j+2)}.$$ Hence it remains to
calculate
\begin{eqnarray*}
\lefteqn{ \sum_{u = 0}^{b^a-1} \sum_{v=0 \atop u \neq v}^{b^a-1}
\de^{2 \pi \icomp(\kappa_0 (u_{a-1} - v_{a-1}) + \cdots +
\kappa_{a-1} (u_0 - v_0))} \int_{u/b^a}^{(u+1)/b^a}
\int_{v/b^a}^{(v+1)/b^a} |x-y|^j \rd x \rd y } \\ & = & 2\sum_{u =
0}^{b^a-2} \sum_{v=u+1}^{b^a-1} \de^{2 \pi \icomp (\kappa_0 (u_{a-1}
- v_{a-1}) + \cdots + \kappa_{a-1} (u_0 - v_0))}  \frac{2
j!}{b^{a(j+2)}} \sum_{i=0}^{\lfloor j/2 \rfloor}
\frac{|u-v|^{j-2i}}{(j-2i)!(2i+2)!} \\ & = & \frac{4 j!}{b^{a(j+2)}}
\sum_{i=0}^{\lfloor j/2 \rfloor} \frac{1}{(j-2i)!(2i+2)!} \sum_{u =
0}^{b^a-2} \sum_{v=u+1}^{b^a-1} \frac{\de^{2 \pi \icomp (\kappa_0
(u_{a-1} - v_{a-1}) + \cdots + \kappa_{a-1} (u_0 -
v_0))}}{(v-u)^{2i-j}},
\end{eqnarray*}
where we used Lemma~\ref{lem_int}. The absolute value of the inner
double sum can now be bounded using Lemma~\ref{lem_doublesum2} and
hence the result follows.
\end{proof}

\begin{lemma}\label{lem_sumr}
Let $b\ge 2$ be an integer and let $\alpha > 1/2$ be a real number.
Then we have $$\sum_{k=1}^\infty r_{b,\alpha}(k) = 2
\zeta(2\alpha),$$ where $\zeta(2\alpha) = \sum_{h=1}^\infty
h^{-2\alpha}$.
\end{lemma}

\begin{proof}
Let $h \in \integer\setminus\{0\}$ and let  $f_h(x) =
\de^{2\pi\icomp h x}$. The Walsh coefficients $\hat{f}_h(k)$ of the
function $f_h$ are then given by $\hat{f}_h(k) = \int_0^1 f_h(x)
\overline{\wal_k(x)} \rd x$. It follows that $|\hat{f}_h(k)|^2 =
|\beta_{h,k}|^2$, where $\beta_{h,k}$ was defined in
Lemma~\ref{lem_rgen}. Using Parseval's equality we obtain
$$\sum_{k=1}^\infty |\beta_{h,k}|^2 = \sum_{k=1}^\infty
|\hat{f}_h(k)|^2 = \int_0^1 |f_h(x)|^2\rd x = \int_0^1 1 \rd x =
1.$$ Hence we have $$\sum_{k=1}^\infty r_{b,\alpha}(k) = \sum_{h\in
\integer\setminus\{0\}} \frac{1}{|h|^{2\alpha}} \sum_{k=1}^\infty
|\beta_{h,k}|^2 = \sum_{h\in \integer\setminus\{0\}}
\frac{1}{|h|^{2\alpha}} = 2\zeta(2\alpha).$$ The result follows.
\end{proof}

\begin{lemma}\label{lem_r2}
Let $k = \kappa_{a_1-1} b^{a_1-1} + \cdots +
\kappa_{a_\nu-1}b^{a_\nu-1}$ with $1 \le a_\nu < \cdots < a_1$ and
let $\kappa_{a_1-1},\ldots, \kappa_{a_\nu-1} \in \{1,\ldots, b-1\}$.
Then
\begin{equation*}
\beta_{h,\kappa_{a_1-1}b^{a_1-1} + \cdots +
\kappa_{a_\nu-1}b^{a_\nu-1}} =  \sum_{h_1,\ldots, h_\nu \in\integer,
h_l \equiv \kappa_{a_l-1}\pmod{b} \atop h = h_1 b^{a_1-1} + \cdots +
h_\nu b^{a_\nu -1}}\frac{b^{\nu}}{(2\pi\icomp)^\nu} \prod_{l=1}^\nu
\frac{1-\de^{2\pi\icomp h_l/b}}{h_l}.
\end{equation*}
\end{lemma}

\begin{proof}
First we consider $k = \kappa_{a-1} b^{a-1}$ with $\kappa_{a-1} \in
\{1,\ldots, b-1\}$. Let $x = \frac{x_1}{b} + \frac{x_2}{b^2} +
\cdots$, then we have $\wal_k(x) = \de^{2\pi \icomp \kappa_{a-1}
x_{a}/b}$. Note that $\wal_k(x)$ is constant in the intervals
$[u/b^{a}, (u+1)/b^{a})$ for $0 \le u < b^{a}$. Let $u = u_{a-1}
b^{a-1} + \cdots + u_0$. Then for any $h \in \integer\setminus\{0\}$
we have
\begin{eqnarray*}
\beta_{h,k}
&=&  \sum_{u=0}^{b^{a}-1} \de^{2 \pi \icomp \kappa_{a-1} u_{0}/b} \int_{u/b^{a}}^{(u+1)/b^{a}} \de^{-2\pi \icomp h x} \rd x \\
&=&  \sum_{u=0}^{b^{a}-1} \de^{2 \pi \icomp \kappa_{a-1} u_{0}/b}
\frac{\de^{-2\pi \icomp h (u+1)/b^{a}} - \de^{-2\pi \icomp h u/b^{a}}}{-2\pi \icomp h} \\
&=& \frac{1-\de^{-2\pi \icomp h/b^{a}}}{2\pi h \icomp}
\sum_{u_0=0}^{b-1} \cdots \sum_{u_{a-1} =0}^{b-1}
\de^{2 \pi \icomp \kappa_{a-1} u_{0}/b} \de^{-2\pi \icomp h (u_{a-1}/b + \cdots + u_{0}/b^{a})} \\
&=& \frac{1-\de^{-2\pi \icomp h/b^{a}}}{2\pi h \icomp}
\sum_{u_0=0}^{b-1} \de^{2\pi \icomp u_0 (\kappa_{a-1}/b - h/b^{a})}
\sum_{u_1=0}^{b-1} \de^{-2\pi \icomp u_1 h/b^{a-1}} \cdots \!\!\!
\sum_{u_{a-1} =0}^{b-1} \! \de^{-2 \pi \icomp u_{a-1} h/b}.
\end{eqnarray*}
Let now $h \in \integer\setminus\{0\}$ and let $h = h_{c-1} b^{c-1}
+ \cdots + h_0$ and set $h_c = h_{c+1} = \cdots = 0$. If $h > 0$ we
assume that $h_i \in \{0,\ldots, b-1\}$ and if $h < 0$ we assume
that $h_i \in \{-b+1,\ldots, 0\}$ for all $i \ge 0$. If $h_0 \neq 0$
then $\sum_{u_{a-1} =0}^{b-1} \de^{-2 \pi \icomp u_{a-1} h/b} = 0$
and hence $\beta_{h,\kappa_{a-1} b^{a-1}} = 0$. If $h_0 =0$ then
$\sum_{u_{a-1} =0}^{b-1} \de^{-2 \pi \icomp u_{a-1} h/b} = b$. In
general, if for an $0 \le i < a-1$ we have $h_i \neq 0$ then
$\beta_{h,\kappa_{a-1} b^{a-1}} = 0$. Further, if $h_i = 0$ for $0
\le i \le a-1$ then we also have $\beta_{h,\kappa_{a-1} b^{a-1}} =
0$. Hence, in order to obtain $\beta_{h,\kappa_{a-1} b^{a-1}} \neq
0$ we must have $h_0 = \cdots = h_{a-2} = 0$ and $\kappa_{a-1} -
h_{a-1} \equiv 0 \pmod{b}$. In this case we have $$\beta_{h,
\kappa_{a-1}b^{a-1}} = \frac{1-\de^{-2\pi \icomp h_{a-1}/b}}{2\pi h
\icomp} b^{a},$$  where $h = h_{a-1} b^{a-1} + h_{a}b^{a} + \cdots$
with $h_{a-1} \equiv \kappa_{a-1} \pmod{b}$. We can also write
$$\beta_{h b^{a-1},\kappa_{a-1}b^{a-1}} = \frac{b(1-\de^{-2\pi
\icomp h/b})}{2\pi \icomp h},$$ with $h \in \integer$ such that $h
\equiv \kappa_{a-1} \pmod{b}$.

We can interpret $\beta_{h,k} = \int_0^1 \de^{-2\pi \icomp h x}
\wal_k(x) \rd x$ as the Fourier coefficients of the $k$-th Walsh
function, hence it follows that $$\wal_k(x) = \sum_{h \in \integer}
\beta_{h,k} \de^{2\pi \icomp h x}.$$

Let now $k = \kappa_{a_1-1} b^{a_1-1} + \cdots + \kappa_{a_\nu-1}
b^{a_\nu-1}$ for some $1 \le a_\nu < \cdots < a_1$. Then we have
\begin{eqnarray*}
\lefteqn{\wal_{\kappa_{a_1-1}b^{a_1-1} + \cdots + \kappa_{a_\nu-1}b^{a_\nu-1}}(x) } \\ &=& \wal_{\kappa_{a_1-1}b^{a_1-1}}(x) \cdots \wal_{\kappa_{a_\nu-1}b^{a_\nu-1}}(x) \\
&=& \sum_{h_1 \in \integer} \beta_{h_1,\kappa_{a_1-1}b^{a_1-1}} \de^{2\pi \icomp h_1 x} \cdots \sum_{h_\nu \in \integer} \beta_{h_\nu,\kappa_{a_\nu-1}b^{a_\nu-1}} \de^{2\pi \icomp h_\nu x} \\
&=& \sum_{h_1,\ldots, h_\nu \in \integer}
\beta_{h_1,\kappa_{a_1-1}b^{a_1-1}} \cdots  \beta_{h_\nu,
\kappa_{a_\nu-1}b^{a_\nu-1}} \de^{2\pi \icomp (h_1 + \cdots + h_\nu)
x}.
\end{eqnarray*}
On the other hand we have $$\wal_{\kappa_{a_1-1}b^{a_1-1} + \cdots +
\kappa_{a_\nu-1}b^{a_\nu-1}}(x) = \sum_{h \in
\integer\setminus\{0\}} \beta_{h,\kappa_{a_1-1}b^{a_1-1} + \cdots +
\kappa_{a_\nu-1}b^{a_\nu-1}} \de^{2\pi \icomp h x}.$$ On comparing
the last two equations we obtain that
$\beta_{h,\kappa_{a_1-1}b^{a_1-1} + \cdots +
\kappa_{a_\nu-1}b^{a_\nu-1}} = 0$ if either $b^{a_1-1}
\not\hspace{-0.1cm}| h$ or $h \not\equiv \kappa_{a_1-1}
\pmod{b^{a_1-1}}$. Now let $h \in \integer$ such that $b^{a_1-1} |
h$ and $h \equiv \kappa_{a_1-1} \pmod{b^{a_1-1}}$. Then we have
\begin{eqnarray*}
\lefteqn{ \beta_{h,\kappa_{a_1-1}b^{a_1-1} + \cdots +
\kappa_{a_\nu-1}b^{a_\nu-1}} } \\ &=& \sum_{h_1,\ldots, h_\nu
\in\integer, h_l \equiv \kappa_{a_l-1}\pmod{b} \atop h = h_1
b^{a_1-1} + \cdots + h_\nu b^{a_\nu -1}} \beta_{h_1
b^{a_1-1},\kappa_{a_1-1}b^{a_1-1}} \cdots  \beta_{h_\nu b^{a_\nu-1},
\kappa_{a_\nu-1}b^{a_\nu-1}}  \\ &=&  \sum_{h_1,\ldots, h_\nu
\in\integer, h_l \equiv \kappa_{a_l-1}\pmod{b} \atop h = h_1
b^{a_1-1} + \cdots + h_\nu b^{a_\nu
-1}}\frac{b^{\nu}}{(2\pi\icomp)^\nu} \prod_{l=1}^\nu
\frac{1-\de^{2\pi\icomp h_l/b}}{h_l}
\end{eqnarray*}
and the result follows.
\end{proof}

\begin{lemma}\label{lem_rbm}
For $k \ge 1$, $b \ge 2$, $m \ge 1$ and $\alpha > 1/2$ we have
$$r_{b,\alpha}(k b^m) = b^{-2\alpha m} r_{b,\alpha}(k).$$
\end{lemma}
\begin{proof}
First note that $\beta_{h,\kappa_{a_1-1}b^{m+a_1-1} + \cdots +
\kappa_{a_\nu-1}b^{m+a_\nu-1}} = 0$ if $b^m \not| h$. Further it
follows from the previous lemma that $$\beta_{h
b^m,\kappa_{a_1-1}b^{m+a_1-1} + \cdots +
\kappa_{a_\nu-1}b^{m+a_\nu-1}} = \beta_{h,\kappa_{a_1-1}b^{a_1-1} +
\cdots + \kappa_{a_\nu-1}b^{a_\nu-1}}$$ and hence by
Lemma~\ref{lem_rgen} we have $$r_{b,\alpha}(k b^m) =  \sum_{h \in
\integer\setminus\{0\}} \frac{|\beta_{h
b^m,kb^m}|^2}{|hb^m|^{2\alpha}} = b^{-2\alpha m} \sum_{h \in
\integer\setminus\{0\}} \frac{|\beta_{h,k}|^2}{|h|^{2\alpha}} =
b^{-2\alpha m} r_{b,\alpha}(k).$$ The result follows .
\end{proof}

\section*{Acknowledgement}
The support of the Australian Research Council under its Centre of
Excellence Program is gratefully acknowledged.

\end{document}